\documentclass[reqno,11pt]{amsart}

\usepackage{amssymb}
\usepackage{amsthm}
\usepackage{amsmath}
\usepackage{graphicx}
\usepackage{comment}
\usepackage[usenames,dvipsnames]{xcolor}
\usepackage{geometry}

\usepackage{tikz}

\usepackage{helvet}

\newtheorem{thm}{Theorem}[section]

\newtheorem{lem}[thm]{Lemma}

\newtheorem{rem}{Remark}[section]

\theoremstyle{definition}

\numberwithin{equation}{section}

\newcommand{\rr}{{\mathbb R}}

\newcommand{\nn}{\textit{{\textbf{n}}}}

\newcommand{\grad}{\nabla}

\newcommand{\prim}{^\prime}

\newcommand{\ep}{\varepsilon}

\newcommand{\limit}{\lim\limits_}
\newcommand{\bs}{\backslash}

\newcommand{\pa}{\partial}
\newcommand{\lap}{\Delta}
\newcommand{\les}{\lesssim}

\newcommand{\epp}{{\ep^\prime}}
\newcommand{\JJ}{\mathcal{J}_\ep}

\begin{document}                        


\title{On the Two-Dimensional Muskat Problem \\ with Monotone Large Initial Data}

\author[F. Deng]{Fan Deng}
\address[F. Deng]{School of Mathematical Sciences; Fudan University, Shanghai 200433, P.R.China.}
\email{fandeng12@fudan.edu.cn}
\author[Z. Lei]{Zhen Lei}
\address[Z. Lei]{School of Mathematical Sciences; LMNS and Shanghai Key Laboratory for Contemporary Applied Mathematics, Fudan University, Shanghai 200433, P. R. China.}
\email{zlei@fudan.edu.cn}
\author[F.H. Lin]{Fanghua Lin}
\address[F.H. Lin]{Courant Institute of Mathematics, New York University, USA; and Institute of Mathematical Sciences of NYU-ECNU at NYU-Shanghai, PRC.}
\email{linf@cims.nyu.edu}
\date{\today}
\maketitle

\begin{abstract}
We consider the evolution of two incompressible, immiscible fluids with different
densities in porous media, known as the Muskat problem \cite{Muskat}, which in two
dimensions is analogous to the Hele-Shaw cell \cite{Taylor}. We establish, for a 
class of large and monotone initial data, the global existence of weak 
solutions. The proof is based on a local well-posedness result for 
the initial data with certain specific asymptotics at spatial infinity and a  
new maximum principle for the first derivative of the graph function.  
\end{abstract}

\section{Introduction}

The Muskat problem \cite{Bear, Muskat, Taylor}, mathematically analogous to the vertical
Hele-Shaw model in two-dimensional case, describes the motion of two incompressible immiscible fluids in porous media, has attracted a great deal of attentions in recent years, see for examples \cite{breakdown,Cordoba,Maximum,Peter,Peter2,Peter3}. In particular, 
\cite{Peter3} gave a very nice summary of earlier results as well
as  some most updated progress on this problem. The interface between the two fluids is
a "vortex sheet" \cite{Caflisch2,Majda}, since the normal velocity is continuous while
the tangential velocity is discontinuous. In the Muskat problem, the fluids motion are
governed by the Darcy's law. Under the assumption that the gravity is the only body
force acting on the fluid, the momentum equation becomes (in two-dimensional case)
\begin{equation}\label{Darcy}
\begin{split}
u=-\frac{\kappa}{\mu}[\grad p + (0,\rho g)].
\end{split}
\end{equation}
Here $u$ is the velocity field, $\rho$ is the fluid density, $p$ is the pressure, $\mu$ is viscosity, $g$ is the acceleration due to gravity, and $\kappa$ is the permeability.

Darcy's Law may be viewed as a substitute for the momentum equation in the Navier-Stokes system. In fact, it can be considered as a statistical average of the latter, see \cite{Muskat}. Naturally there are three-dimensional
correspondences, but in this paper we only consider the two-dimensional case, and the fluids are incompressible:
\begin{equation}\label{incompressible}
\grad \cdot u=0.
\end{equation}
In addition, the conservation of mass gives
\begin{equation}\label{rho}
\rho_t+ u\cdot \grad \rho=0.
\end{equation}
The resulting system of \eqref{Darcy}, \eqref{incompressible} and \eqref{rho} is closed.

For the Muskat problem, we consider in this paper that
\begin{equation}
\rho(x,t)=
\begin{cases}
\rho^1, \qquad x\in\Omega^1(t),\\
\rho^2, \qquad x\in\Omega^2(t)=\rr^2\bs\Omega^1(t).
\end{cases}
\end{equation}
$\rho^1,~\rho^2$ are two different constants correspond to two homogeneous fluids. Often, the viscosity of the two fluids can be different, where the Atwood Number $A_\mu$ is introduced and it is defined by
\begin{equation}\nonumber
A_\mu=\frac{\mu_2-\mu_1}{\mu_1+\mu_2}.
\end{equation}
For the case $A_\mu\neq0$, we refer to \cite{Ambrose,Siegel}. Sometimes the surface
tension is also considered; it has a regularizing effect \cite{Ester}. We shall
consider the case that $A_\mu=0$ and with no surface tension. Then the system can
be transformed into a contour dynamics system \cite{Cordoba}

\begin{equation}
\begin{split}
3D~\text{case}:\  \frac{df(x,t)}{dt}=&\frac{\rho^2-\rho^1}{4\pi} \\
               &\cdot PV\int_{\rr^2} \frac{(\grad f(x,t)-\grad f(x-y,t))\cdot y}{[|y|^2+(f(x,t)-f(x-y,t))^2]^{3/2}}dy, \\[10pt]
2D~\text{case}:\  \frac{df(x,t)}{dt}=&\frac{\rho^2-\rho^1}{2\pi} \\
               &\cdot PV\int_{\rr}\frac{(\pa_x f(x,t)-\pa_x f(x-\alpha,t))\alpha}{\alpha^2+(f(x,t)-f(x-\alpha,t))^2} d\alpha, \\[10pt]
\end{split}
\end{equation}
where $f(x,t)$ is the function which gives a graphical representation of the
interface. In the three-dimensional case, the interface is a two-dimensional (sheet or) surface in $\rr^3$, while in two dimensions it is a curve in the plane. Rayleigh
\cite{Rayleigh} and Saffman-Taylor \cite{Taylor} gave a condition that must be satisfied
for the linearized model in order to ensure local well-posedness, namely the normal
component of the pressure gradient jump at the interface has to have a distinguished
sign, known as the Rayleigh-Taylor condition
\begin{equation}\label{RT}
-(\grad p^2(x,f(x),t)-\grad p^1(x,f(x),t))\cdot \nn > 0,
\end{equation}
where the normal vector of the curve $(x,f(x))$ is given by 
\begin{equation}\nonumber
\nn=\frac{(-f\prim(x),1)}{\sqrt{1+f\prim(x)^2}}.
\end{equation}
Noting that the two fluids are immiscible, we have, at the free interface, that
\begin{equation}\label{INCOMP}
(u^1-u^2)\cdot \nn=0 \ \ \ i.e.\ \ \  [u]\cdot \nn=0.
\end{equation}
Here $p^1$, $p^2$, $u^1$ and     $u^2$ denote the limits approaching the interface. Using \eqref{RT}, \eqref{INCOMP} together with the Darcy's Law \eqref{Darcy}, we obtain
\begin{equation}\label{RT2}
\rho^1<\rho^2,
\end{equation}
which is the Rayleigh-Taylor condition in our case. \eqref{RT2} can be seen also from the 
linearized system, which will be explained it in Section 2.

Let us recall a few results closely related to this paper. Cordoba-Gancedo
\cite{Cordoba} derived the contour dynamics system and proved its local well-posedness
with $H^k$ , for $k$ suitably large, initial data.
They also showed ill-posedness when Rayleigh-Taylor condition is not satisfied. In \cite{Maximum},
they gave a maximum principle satisfied by the
graph function $f(x,t)$ in various cases. Cordoba-Fefferman-Gancedo \cite{breakdown}
showed that in two dimensions a nonempty open set of initial data in $H^4$ develops finite-time singularity when the R-T condition is violated. Recently, Constantin etc. \cite{Peter}
obtained the global existence for the exact small initial data in two dimensions satisfying $\|f_0\|_1 \leq \frac{1}{5}$. Among other things, this result was improved much
further in their recent
papers \cite{Peter2, Peter3} to $\|f_0\|_1 \leq \frac{1}{3}$. We remark that the norm $\|\cdot\|_1$ used in their paper is defined by $\|f\|_1=\int |\xi||\hat{f}(\xi)|d\xi$,
which is stronger than $C^1$-norm. They also proved the existence of global weak solutions
if the initial data satisfies that $\|\pa_x f_0\|_{L^\infty}$ is smaller than a given
constant in both two dimensions and three dimensions. The compactness comes from their observation that $\|\pa_x f\|_{L^\infty(\rr)}<1$ in two dimensions and $\|\pa_x f\|_{L^\infty(\rr^2)}<\frac{1}{3}$ in three dimensions for all time whenever these are valid initially.

Motivated by \cite{Peter}, in this paper, we consider the two-dimensional contour dynamics equation (derived in \cite{Cordoba}) with large initial data:
\begin{equation}\label{OS}
\begin{split}
&f_t(x,t)=\frac{\rho^2-\rho^1}{2\pi} PV\int_{\rr}\frac{(\pa_x f(x,t)-\pa_x f(x-\alpha,t))\alpha}{\alpha^2+(f(x,t)-f(x-\alpha,t))^2}d\alpha ,\\[5pt]
&f(x,0)=f_0(x), \quad \quad x\in \rr.\\[10pt]
\end{split}
\end{equation}
We assume that $f_0(x)$ is monotonically decreasing (the results of this paper also hold for monotonically increasing initial data with a similar argument) and satisfies
\begin{equation}\label{IDB}
\|f_0\|_{L^\infty} < \infty, \quad \|\pa_x f_0\|_{L^\infty} < \infty, \quad \pa_x f_0(x)\leq 0,
\end{equation}
with the asymptotics
\begin{equation}\label{IDA}
\begin{split}
&f_0(x)\rightarrow a, ~as~x\rightarrow -\infty,\\[5pt]
&f_0(x) \rightarrow b, ~as ~ x \rightarrow \infty.   \qquad(a>b\ are\ two\ constants) 
\end{split}
\end{equation}
\begin{rem}
The restriction of monotone initial data here was inspired by the Nickel-Matano-Henry zero
number (or lap number) theory \cite{Henry,Matano1,Nickel}. This theory is based on the
maximum principle (see e.g. \cite{Angenent2,Matano1}) and was employed to study the
long-time behavior of semilinear parabolic equations (cf. \cite{Angenent1,Angenent3,Matano2}).
It tells us basically that solutions preserve the monotone property if the initial data are
monotone; i.e., the lap number does not increase, at least for the 1D semilinear parabolic
equation. And the equation considered here has a linear (nonlocal) parabolic part (see
Section 2), which suggests a nonlocal nonlinear parabolic nature of the problem. We will
discuss this further in Section 6. The zero number theory also suggest us to explore the
asymptotical behavior of our solution, which seems to be quite interesting although we
would not discuss it in this paper. 

To explore the maximum principle for derivatives of graph functions, we adopted a similar
method as in \cite{Maximum} along with several new observations. One of the technical issue
we need to overcome is the fact that $f(x,t)$  does not vanish at spatial infinity (see
Section 4 and 6).
\end{rem}

With the Rayleigh-Taylor stability condition \eqref{RT2}, we may assume the constant $\frac{\rho^2-\rho^1}{2\pi}=1$ to simplify notation.
Our main result is

\begin{thm}\label{MAINTHM}
For given initial data satisfying \eqref{IDB} and \eqref{IDA}, the contour dynamics system \eqref{OS} admits a global in time weak solution
\begin{equation}\nonumber
f(x,t)\in C([0,T]\times\rr)\cap L^\infty([0,T];W^{1,\infty}(\rr)), \qquad \forall ~T>0.
\end{equation}
Moreover the solution is also monotone for all time in the spatial direction.
\end{thm}
We remark that $\dot{W}^{1,\infty}(\rr)$ is the critical space for the graph function in
equation \eqref{OS} (see further discussions in Section 2). Because the
integrodifferential equation \eqref{OS} has a strong singularity, it generates a loss
of derivatives in estimates, the boundness of $\dot{W}^{1,\infty}(\rr)$ of the graph function is
hence a key to establish global solutions for the contour dynamics system
\eqref{OS}.  It should be noted that previous global results \cite{Peter,Peter2} rely on
a smallness of the critical norm. In the present paper, we explored a new maximum principle
which may lead to estimates that exceed those from scaling invariants of the equation
\eqref{OS}. Since this maximum principle is valid for monotone initial data, we need to
establish first the following local well-posedness theorem with given spatial asymptotics
and with infinite energy:

\begin{thm}\label{LW}
For given initial data satisfying \eqref{LIDA} and \eqref{LIDB}, there are a time $T>0$
and a unique solution
\begin{equation}
f(x,t)\in C([0,T],C^{2,\gamma}(\rr)\cap\dot{H}^1(\rr)\cap \dot{H}^3(\rr)) \\
\end{equation}
to the contour dynamics system \eqref{OS}.
\end{thm}

Theorem \ref{LW} does not seem to follow from the results of \cite{Cordoba}. The local
well-posedness results in \cite{Cordoba} require initial data to be at least in energy
spaces. If we cut off the initial data such as in \cite{Peter}, the regularized initial
data would then not maintain the monotone property, which is however crucial for our
result.  To obtain local solutions, we use an approximate approach as in \cite{Peter}. To start 
with, we add an artificial viscous term $"\epp f^\epp_{xx}"$. One
soon realizes that such a term can not overcome the possible derivative losses of the
system \eqref{OS}. We then regularize the nonlinear part via the method inspired by
\cite{Peter}. In the process of making
\textit{a priori} estimates for local solutions (in Section 4), we do not have the natural $L^2$ bound
as in \cite{Cordoba} (We use the homogeneous Sobolev spaces partly because of it), which is crucial in \cite{Cordoba}. Instead, we use the $L^\infty$ norm coming from the maximum principle for $f(x,t)$.

Moreover, we show that the local solutions enjoy the same spatial asymptotics as the
initial data. The latter is possible because the integrodifferential equation
\eqref{OS} have a special structure: the terms of \eqref{OS} involve only the
derivatives and the differences of the unknown function, in particular, constant functions
are steady states.
We therefore add some restrictions to the standard H\"{o}lder Spaces such that the solution
space becomes $C^{j,\gamma}(\rr)\cap \dot{H}^k(\rr)(k\geq 1)$. (The definition of the homogeneous Sobolev seminorm can be found below.)

The rest of this paper are organized as follows. In section 2 we recall some related
properties of the system \eqref{OS}. Then we propose an approximate scheme for \eqref{OS}
and prove its local well-posedness in Section 3. In Section 4 and 5, we establish local
well-posedness for the contour dynamics system \eqref{OS}. And finally, we show the
existence of global weak solutions in Section 6.\\

\textit{Notations} \quad We use $\Delta_\alpha f(x,t)$ to denote the standard difference of $f$: \\
\begin{equation*}
 \Delta_\alpha f(x,t)=\frac{f(x)-f(x-\alpha)}{\alpha},\qquad \Delta_{x-\alpha}f(x,t)=\frac{f(x)-f(\alpha)}{x-\alpha}.\\[5pt]
\end{equation*}
 We use $|\cdot|_{\dot{C}^\gamma}$ to denote the homogenous H\"{o}lder seminorm with the index $0<\gamma<1$:
\begin{equation*}
|f(\cdot)|_{\dot{C}^\gamma}=\sup_{x\neq y}\frac{|f(x)-f(y)|}{|x-y|^\gamma}.
\end{equation*}

$\|\cdot\|_{C^\gamma}=\|\cdot\|_{L^\infty}+|\cdot|_{\dot{C}^\gamma}$ represents the inhomogeneous H\"{o}lder norm. We use $\|\cdot\|_{\dot{H}^k}$ to denote the standard homogeneous Sobolev seminorm: 
$\|f\|_{\dot{H}^k}=\|\sum_\beta D^\beta f\|_{L^2}$, where $\beta 's$ is the multi-index with $|\beta|=k$. We use $A \les B$ to denote that $A\leq C_0 B$ and $A\simeq B$ to denote that $C_0^{-1}B \leq A \leq C_0 B$ for some constant $C_0>1$ and two positive quantities $A$ and $B$.

\section{Preliminaries}

As in \cite{Cordoba}, the contour dynamics equation(CDE) \eqref{OS} can be linearized
around the flat solution. Using the Hilbert Transform \cite{Stein}, we can write
\begin{equation*}
f_t(x,t)=-\frac{\rho^2-\rho^1}{2}\frac{1}{\pi}PV\int_{\rr}\frac{f_x(x-\alpha)}{\alpha}dx=-\frac{\rho^2-\rho^1}{2}Hf_{x}.
\end{equation*}
Applying the Fourier transform, we obtain
\begin{equation*}
\hat{f_t}(\xi)=-\frac{\rho^2-\rho^1}{2}|\xi| \hat{f}(\xi).
\end{equation*}
Therefore, the linearized system reads
\begin{equation}
f_t=-\frac{\rho^2-\rho^1}{2}\Lambda f ,\quad\quad\quad x\in \rr,
\end{equation}
where $\Lambda=\sqrt{-\lap}$. In the three-dimensional case, the linearized system reads
\begin{equation}
f_t=-\frac{\rho^2-\rho^1}{2}(R_1 \pa_{x_1}f+R_2 \pa_{x_2}f)=-\frac{\rho^2-\rho^1}{2}\Lambda f , \quad x\in \rr^2,
\end{equation}
where $R$ denotes the standard Riesz operator \cite{Stein}. Clearly, from the
linearized equations, the nonlinear system \eqref{OS} is stable and well-posed if $\rho^1 < \rho^2$ and unstable if $\rho^1 > \rho^2$, which can also be deduced from
Rayleigh-Taylor Condition \cite{Rayleigh} \cite{Taylor}. Moreover, it is proved in
\cite{Cordoba} that even for arbitrary small initial data in $H^s$ some solutions
leave the space $H^s$ right away in the unstable situation. 

System \eqref{OS} is invariant under the scaling
\begin{equation}\nonumber
\begin{split}
f(x,t) \rightarrow f^\lambda(x,t)=\lambda^{-1} f(\lambda x, \lambda t).
\end{split}
\end{equation}
One can verify that $\sup_{t}\|\pa_x f(\cdot,t)\|_{L^\infty}$ and $\sup_{t} \| \xi
\hat{f}(\xi,t)\|_{L^1} $ are invariant under the above scaling transformation.

System \eqref{OS} enjoys the $L^2$ energy law derived in \cite{Peter}
\begin{equation}\label{EI1}
\begin{split}
\|f\|&_{L^2}^2(t) \\[5pt]
&+ \frac{\rho^2-\rho^1}{2\pi}\int_0^t \int_\rr \int_\rr \ln(1+(\frac{f(x,s)-f(\alpha,s)}{x-\alpha})^2) dxd\alpha ds \\[5pt]
&\ \ \ \ \ =\|f_0\|_{L^2}^2.
\end{split}
\end{equation}
The same energy identity also holds for the three-dimensional case \cite{Peter2}. Unfortunately, in \eqref{EI1},  the second term on the left-hand side does not provide a gain of a half derivative for $\rho^2>\rho^1$ as it is in the linear case. The energy identity in the latter case is given by
\begin{equation}\nonumber
\begin{split}
\|f\|_{L^2}^2(t)+ \frac{\rho^2-\rho^1}{2\pi}\int_0^t \int_\rr \int_\rr (\frac{f(x,s)-f(\alpha,x)}{x-\alpha})^2 dxd\alpha ds =\|f_0\|_{L^2}^2.
\end{split}
\end{equation}
However, when the first derivative of $f(x,t)$ is small under the norm $L^\infty$,
the Taylor expansion of $\ln(1+u^2)$ would lead to a gain in estimates on
derivatives, and consequently one obtains a compactness. This is one of
the keys in \cite{Peter} \cite{Peter2} where they established global solutions evolved from
an exact class of small initial data. In our case, the compactness comes from a new
maximum principle in Section 6.

\section{Approximate System for Constructing Local Solutions}

In this section, we propose the approximate scheme for the contour dynamics system
\eqref{OS} and establish local well-posedness for this approximate system. This will be
used to construct local solutions to \eqref{OS} in Section 4 and 5. We will focus on
initial data (see Figure 3.1) which has also spatial limits as $x$ goes to infinity
\begin{equation}\label{LIDA}
\begin{split}
&\qquad\quad \bar{f}_0(x)\rightarrow a \quad as\ x\rightarrow -\infty,\\[5pt]
&\qquad\quad \bar{f}_0(x) \rightarrow b \quad as\ x \rightarrow \infty,   \qquad(a>b\ are\ two\ constants) 
\end{split}
\end{equation}
along with certain smoothness,
\begin{equation}\label{LIDB}
\bar{f}_0(x)\in C^{2,\gamma}(\rr)\cap \dot{H}^1(\rr)\cap \dot{H}^3(\rr).\\
\end{equation}

Note that the monotonicity is not required here. The homogeneous Sobolev spaces $\dot{H}^1(\rr)$ and $\dot{H}^3(\rr)$ here are explained and defined in Section 1. 

We consider the following approximate system:
\begin{equation}\label{LAS}
\begin{cases}
& \pa_t f^\epp(x,t)-\epp f^\epp_{xx}= \widetilde{F}(f^\epp),\\
& f^\epp(x,0)= \bar{f}_0(x),
\end{cases}
\end{equation}
where
\begin{equation}\label{FFF}
\begin{split}
 \widetilde{F}(f^\epp)&=: \pa_x PV\int_\rr \arctan(\Delta_\alpha^\epp f^\epp(x,t)) d\alpha \\[5pt]
                      &= PV\int_\rr \pa_x \Delta_\alpha^\epp f^\epp d\alpha  -PV\int_{\rr}\frac{\pa_x \Delta_\alpha^\epp f^\epp (\Delta_\alpha^\epp f^\epp)^2 }{1+(\Delta_\alpha^\epp f^\epp)^2} d\alpha, \\[15pt]
\end{split}
\end{equation}

\begin{center}\nonumber \label{fig1}
\begin{tikzpicture}[scale=0.5]
\draw [thick](-4.5,3.03) to[out=0,in=179.6] (-3, 3)to [out=0.4,in=135](-2,2.5)to[out=-45,in=180](-1,2) to [out=0, in=-110](0,4) to [out=70, in=180](0.6,4.5)
to [out=0,in=130](1.3,4)to [out=-50, in=170](4,0) to [out=-10, in=180](6,0.8) to [out=0,in=180.5](6.5,0.83) to[out=0.5, in=181] (7.5,0.85);
\node [below] at (1.5, -1){$\mathit{FIGURE~3.1.}$ Initial Data ~$\bar{f}_0(x)$};\label{fig1}
\end{tikzpicture}
\end{center}
and $\Delta_\alpha^{\epp}$, $\Delta_{x-\alpha}^{\epp}$ denote the regularized difference operators (note that \eqref{difference} was from \cite{Peter})
\begin{equation}\label{difference}
\begin{split}
\Delta_\alpha^\epp f^\epp= \frac{f^\epp(x)-f^\epp(x-\alpha)}{\frac{\alpha}{|\alpha|^\epp}}, \quad \Delta_{x-\alpha}^\epp f^\epp = \frac{f^\epp(x)-f^\epp(\alpha)}{\frac{x-\alpha}{|x-\alpha|^\epp}}.
\end{split}
\end{equation}

For the rest of this section, we write $f(x,t)\ =\ f^\epp(x,t)$ to simplify notation. We claim that we can write the first term in \eqref{FFF} as $-(1-\epp)\Lambda^{1-\epp}~f(x)$. In fact, as in \cite{QG}, when $\|f(\cdot, t)\|_{L^\infty}$ and $\|\pa_x f(\cdot, t)\|_{L^\infty}$ are bounded, we can compute that
\begin{equation}\label{LinearPart}
\begin{split}
PV\int_\rr& \pa_x \Delta_\alpha^\epp f d\alpha \\[5pt]
= & -PV\int_\rr \frac{\pa_\alpha (f(x)-f(x-\alpha))}{\frac{\alpha}{|\alpha|^\epp}} d\alpha \\[5pt]
                                              = & -\limit{\delta \rightarrow 0} \int_{\delta \leq |\alpha| \leq \frac{1}{\delta}}
\frac{\pa_\alpha (f(x)-f(x-\alpha))}{\frac{\alpha}{|\alpha|^\epp}} d\alpha \\[5pt]
                                              = & -\limit{\delta \rightarrow 0} |\delta|^\epp \frac{f(x)-f(x-\alpha)}{\alpha}\Big{|}_{|\alpha|=\delta}-\limit{\delta \rightarrow 0}\frac{f(x)-f(x-\alpha)}{\frac{\alpha}{|\alpha|^\epp}}\Big{|}_{|\alpha|=\frac{1}{\delta}} \\[5pt]
                                                &+(\epp-1)\limit{\delta \rightarrow 0}\int_{\delta \leq |\alpha| \leq \frac{1}{\delta}} \frac{f(x)-f(x-\alpha)}{|\alpha|^{2-\epp}}d\alpha  \\[5pt]
                                              = & -(1-\epp) PV\int_\rr \frac{f(x)-f(x-\alpha)}{|\alpha|^{2-\epp}}d\alpha = -(1-\epp) \Lambda^{1-\epp}f(x).
\end{split}
\end{equation}
Define the solution map in terms of the heat kernel for the system \eqref{LAS} as follows:
\begin{equation}
Sf=e^{\epp t \Delta}\bar{f}_0+\int_0^t e^{\epp(t-s)\Delta}F(f(s))ds ,
\end{equation}
We now set up the problem so that we can use the standard contraction mapping theory. Let $B=L^\infty([0,T]; C^{2,\gamma}(\rr))$ be the Banach space. Let $M>0$ be a constant. Let 
\begin{equation}
E=\{f\in B ~~| ~~\sup_{0\leq t\leq T}\|f(\cdot,t)\|_{C^{2,\gamma}} \leq 2M\}
\end{equation}
be equipped with the norm
\begin{equation}
\|f\|_E = \sup_{0\leq t\leq T}\|f(\cdot,t)\|_{C^{2,\gamma}}.
\end{equation}
It is easy to see that $E$ is a closed convex subset of the Banach space $B$. Then the local well-posedness of the approximation system in \eqref{LAS} is stated in the following theorem:
\begin{thm}\label{LS}
For any fixed $\epp>0$ and given $\bar{f}_0$ satisfying $\|\bar{f}_0\|_E \leq M$, system \eqref{LAS} admits a unique local solution $f^\epp(x,t)$ in $B$. Moreover, if additionally $\|\bar{f}_0(\cdot)\|_{\dot{H}^k}\leq M$ for some $k \geq 1$, then the solution is also in $C([0,T],\dot{H}^k(\rr))$.
\end{thm}
\begin{rem}\label{2.3}
At the end of this section, we shall derive from Theorem \ref{LS} that the local
solutions have the same spatial asymptotics at infinity as
the initial data. That is to say, as $x$ goes to infinity, $f(x,t)$ goes to the same constants as $\bar{f}_0(x)$ for
any~$t\in[0,T]$. The latter is needed in the proof of Theorem \ref{MAINTHM}.
\end{rem}
\begin{proof}
\textit{Step 1.  S maps E to E.} \ \  We need to prove that $\|Sf\|_E \leq 2M$.
We proceed with first the following:
\begin{equation}\label{3.88}
\begin{split}
\sup_{0\leq t\leq T}\|Sf\|_{L^{\infty}}
       \leq & \|\bar{f}_0\|_{L^{\infty}} + \sup_{t} \int_0^t\|e^{\epp(t-s)\Delta}F(f(s))\|_{L^{\infty}}ds   \\[5pt]
       \leq & M + \sup_{t} \int_0^t\|K^\epp(\cdot,t-s)\ast F(f(\cdot,s))\|_{L^{\infty}}ds   \\[5pt]
       \leq & M + CT\sup_{t}\|F(f)\|_{L^{\infty}},
\end{split}
\end{equation}
where $K^\epp(x,t)=\frac{1}{\sqrt{4\pi \epp t}}e^{-\frac{x^2}{4\epp t}}$ is the standard 1D heat kernel. $\|F(f)\|_{L^\infty}$ is estimated by splitting the integral as follows:
\begin{equation}\label{3.99}
\begin{split}
\|F(f)\|_{L^\infty}
       \leq & \|\Lambda^{1-\epp}f\|_{L^{\infty}} + \| \pa_x^2 f\|_{L^{\infty}}\int_{|\alpha|<1} |\alpha|^{\epp}d\alpha \\[5pt]
            & +\|f\|^2_{L^{\infty}}  \|\pa_x f\|_{L^{\infty}}\int_{|\alpha|>1}\frac{1}{|\alpha|^{3-3\epp}}d\alpha  \\[5pt]
       \leq & C(\|f\|_{C^{2,\gamma}}+\|f\|_{C^{2,\gamma}}+\|f\|^3_{C^{2,\gamma}}) \\[5pt]
       \leq & CM^3.
\end{split}
\end{equation}
Then by \eqref{3.88} and \eqref{3.99} we obtain
\begin{equation}
\begin{split}
\sup_{0\leq t\leq T}\|Sf\|_{L^{\infty}}
       \leq  M + CTM^3  \leq  2M .
\end{split}
\end{equation}
In the last inequality above we have chosen $T$ small enough so that $CTM^2$ is smaller than 1. 

Next, we estimate the H\"{o}lder seminorm.
\begin{equation}
\begin{split}
\sup_{0\leq t\leq T}|Sf|_{\dot{C}^\gamma}
       &\leq |\bar{f}_0|_{\dot{C}^\gamma} + \sup_{t} \int_0^t|K^\epp(\cdot,t-s)\ast F(f(\cdot,s))|_{\dot{C}^\gamma}ds   \\[5pt]
       &\leq M + \sup_{t}\int_0^t\|K^\epp(\cdot,t-s)\|_{L^1} |F(f(\cdot,s))|_{\dot{C}^\gamma}ds   \\[5pt]
       &\leq M + CT\sup_{t}|F(f(\cdot,t))|_{\dot{C}^\gamma}  \\[5pt]
       &\leq M + CT(\sup_{t}|\Lambda^{1-\epp}f|_{\dot{C}^\gamma}+|\int_{|\alpha|<1}+\int_{|\alpha|>1}|_{\dot{C}^\gamma})  \\[5pt]
       &\leq M + CT M^{3} \\[5pt]
       &\leq 2M    \quad \quad \quad (Choosing~~~ T~~~ small~~~ enough).
\end{split}
\end{equation}
The homogeneous H\"{o}lder seminorm of the high derivatives can be estimated as follows:

\begin{equation}
\begin{split}
\sup_{0\leq t\leq T}|\pa_x^2 Sf|_{\dot{C}^\gamma}
       &\leq |\pa_x^2 \bar{f}_0|_{\dot{C}^\gamma} + \sup_{t} \int_0^t |\pa_x K^\epp(\cdot,t-s)\ast \pa_x F(f(\cdot,s))|_{\dot{C}^\gamma}ds   \\[5pt]
       &\leq M + C\int_0^t \frac{1}{\sqrt{\epp s}}ds \| \pa_x K^\epp\|_{L^1} \sup_{t}|\pa_x F(f)|_{\dot{C}^\gamma}  \\[5pt]
       &\leq M + C(\epp)\sqrt{t} \sup_{T}|\pa_x F(f(t))|_{\dot{C}^\gamma}.
\end{split}
\end{equation}
After splitting $\sup_{t}|\pa_x F(f(t))|_{\dot{C}^\gamma(\rr)}$ (using its definition)
into various terms, the most difficult term to control is the following $S_1$ in \eqref{SingularTerm}, in which all differentiations are applied to one of
the $f$'s. For the convenience, we shall present only the estimate for $S_1$:
\begin{equation}\label{SingularTerm}
\begin{split}
S_1 =& \sup_{t}|\Lambda^{1-\epp}\pa_x f|_{\dot{C}^{\gamma}} + \sup_{t}|PV\int_{\rr}\frac{\pa_x^2 \Delta_\alpha^\epp f(x) (\Delta_\alpha^\epp f(x))^2}{1+(\Delta_\alpha^\epp f(x))^2}d\alpha|_{\dot{C}^\gamma(\rr)} \\[5pt]
    \leq &CM + \sup_{x\neq y}\frac{1}{|x-y|^{\gamma}}|\int_{\rr}\frac{\pa_x^2 \Delta_\alpha^\epp f(x) (\Delta_\alpha^\epp f(x))^2}{1+(\Delta_\alpha^\epp f(x))^2}d\alpha \\[5pt]
    &-\int_{\rr}\frac{\pa_x^2 \Delta_\alpha^\epp f(y) (\Delta_\alpha^\epp f(y))^2}{1+(\Delta_\alpha^\epp f(y))^2}d\alpha| \\[5pt]
    \leq &CM + Y_1+ Y_2 + Y_3 ,
\end{split}
\end{equation}
where we have also used the following estimate:
\begin{equation}\nonumber
\begin{split}
 &|\Lambda^{1-\epp}\pa_x f|_{\dot{C}^{\gamma}(\rr)}\\[5pt]
 =& \sup_{x\neq y}\frac{1}{|x-y|^{\gamma}} |PV\int_\rr \frac{f_x(x)-f_x(x-\alpha)}{|\alpha|^{2-\epp}}d\alpha \\[5pt]
              & - PV\int_\rr \frac{f_y(y)-f_y(y-\alpha)}{|\alpha|^{2-\epp}}d\alpha| \\
              \leq & (\sup_{x\neq y}\frac{|f_x(x)-f_y(y)|}{|x-y|^{\gamma}}+\sup_{x\neq y}\frac{|f_x(x-\alpha)-f_y(y-\alpha)|}{|x-y|^{\gamma}})\int_{|\alpha|>1} \frac{1}{|\alpha|^{2-\epp}}d\alpha \\[5pt]
              &+ \sup_{x\neq y}\frac{|\int_0^1 \Big{(} f_{xx}(x-\alpha+\theta\alpha)-f_{yy}(y-\alpha+\theta\alpha) \Big{)} d\theta|}{|x-y|^{\gamma}}
              \int_{|\alpha|<1}\frac{1}{|\alpha|^{1-\epp}}d\alpha  \\[5pt]
              \leq & C(\epp)(|\pa_x f|_{\dot{C}^\gamma}+ |\pa_x^2 f|_{\dot{C}^\gamma}) \\[5pt]
              \leq & C(\epp)M.
\end{split}
\end{equation}
And $Y_1,Y_2,Y_3$ in \eqref{SingularTerm} are given by
\begin{equation}
\begin{split}
Y_1 &=\sup_{x\neq y}\frac{1}{|x-y|^{\gamma}}|\int_{\rr}\frac{\pa_x^2 \Delta_\alpha^\epp (f(x)-f(y)) (\Delta_\alpha^\epp f(x))^2}{1+(\Delta_\alpha^\epp f(x))^2}d\alpha|, \\[5pt]
Y_2 &=\sup_{x\neq y}\frac{1}{|x-y|^{\gamma}}|\int_{\rr}\frac{\pa_x^2 \Delta_\alpha^\epp f(y) \Delta_\alpha^\epp (f(x)-f(y))\Delta_\alpha^\epp (f(x)+f(y))}{1+(\Delta_\alpha^\epp f(x))^2}d\alpha|, \\[5pt]
Y_3 &=\sup_{x\neq y}\frac{1}{|x-y|^{\gamma}}|\int_{\rr}\frac{\pa_x^2 \Delta_\alpha^\epp f(y)(\Delta_\alpha^\epp f(y))^2\Delta_\alpha^\epp (f(x)-f(y))\Delta_\alpha^\epp (f(x)+f(y))}{(1+(\Delta_\alpha^\epp f(x))^2)(1+(\Delta_\alpha^\epp f(y))^2)}   d\alpha|. 
\end{split}
\end{equation}
Thanks to the fact that nonlinear terms have already been smoothed out in \eqref{FFF}, we
can estimate $Y_1$ as follows:
\begin{equation}\nonumber
\begin{split}
Y_1 \leq & \sup_{x\neq y}\frac{1}{|x-y|^{\gamma}}\int_{|\alpha|>1}\frac{|\pa_x^2f(x)-\pa_x^2 f(y)-(\pa_x^2f(x-\alpha)-\pa_x^2f(y-\alpha)) ||\Delta_\alpha^\epp f(x)|^2}{\frac{\alpha}{|\alpha|^\epp}} \\[5pt]
     &+   \sup_{x\neq y}\frac{1}{|x-y|^{\gamma}}\int_{|\alpha|<1}\frac{|\pa_x^2f(x)-\pa_x^2 f(y)-(\pa_x^2f(x-\alpha)-\pa_x^2f(y-\alpha) |}{\frac{\alpha}{|\alpha|^\epp}}    \\[5pt]
     \leq & C \|f\|_{L^\infty}^2 |\pa_x^2 f|_{\dot{C}^\gamma}\int_{|\alpha|>1} \frac{1}{|\alpha|^{3-3\epp}}d\alpha
    + C|\pa_x^2 f|_{\dot{C}^{\gamma}}\int_{|\alpha|<1} \frac{1}{|\alpha|^{1-\epp}}d\alpha   \\[5pt]
    \leq  & C(\epp)\|f\|_{C^{2,\gamma}}^3.
\end{split}
\end{equation}
In a similar way, $Y_2$ and $Y_3$ can be bounded by
\begin{equation}\nonumber
Y_2+Y_3 \leq C(\epp) \|f\|_{C^{2,\gamma}}^{3}.
\end{equation}
Therefore we conclude that
\begin{equation}\nonumber
\begin{split}
\sup_{0\leq t\leq T}|\pa_x^2 Sf|_{\dot{C}^\gamma}  \leq M + \sqrt{T} C(\epp) M^{3}  \leq 2M.
\end{split}
\end{equation}
Collecting all the estimates, we have
\begin{equation}\nonumber
\|Sf\|_E \leq \|f\|_E,
\end{equation}
which means that $S$ maps $E$ to $E$.

\textit{Step 2.  S is Contractive.} \quad We need to show that, there exists a $0<\lambda<1$ such that for any $f,g\in E$, there holds
\begin{equation}\label{2.2}
\|Sf-Sg\|_E \leq \lambda \|f-g\|_E.
\end{equation}
Similar to arguments in \textit{Step 1}, the proof is also straightforward. For instance,
for the term with highest derivative $|\pa_x^2 (Sf-Sg)|_{\dot{C}^\gamma(\rr)}$, we
can compute that
\begin{equation}\nonumber
\begin{split}
|\pa_x^2(Sf-Sg)|_{\dot{C}^{\gamma}(\rr)} &= |\pa_x^2 \int_0^t e^{\epp(t-s)\Delta}(F(f)-F(g))(s)ds |_{\dot{C}^{\gamma}(\rr)}  \\[5pt]
                                   &\leq |\int_0^t\int_{\rr}\pa_x K^\epp(x-y,t-s)\pa_x(F(f)-F(g))(y,s) dyds|_{\dot{C}^{\gamma}(\rr)}  \\[5pt]
                                   &\leq \sqrt{T}C(\epp) \sup_{t}|\pa_x(F(f)-F(g))(\cdot,t) |_{\dot{C}^{\gamma}(\rr)}. 
\end{split}
\end{equation}
The last term can be estimated as that for $Y_1$, and we omit these details.

In short, we can derive
\begin{equation}
\|Sf-Sg\|_E\leq \sqrt{T} C(\epp) \|f-g\|_E.
\end{equation}
By choosing $T$ sufficiently small, \eqref{2.2} follows.

\textit{Step 3. $\dot{H}^k(\rr)$-Preserving Property.}\quad  Let the initial data $\bar{f}_0(x)$ be in $\dot{H}^k(\rr)$ for some $k \geq 1$ with $\|\bar{f}_0\|_{\dot{H}^k}\leq M$, we then modify the definition of the closed subset $E$ to be
\begin{equation}\nonumber
E=\{f\in B ~~| ~~\sup_{0\leq t\leq T}\|f(\cdot,t)\|_{C^{2,\gamma}(\rr)} \leq 2M\text{,} ~~\sup_{0\leq t \leq T}\|f(\cdot,t)\|_{\dot{H}^k(\rr)} \leq 2M \}
\end{equation}
equipped with the norm
\begin{equation}
\|f\|_E = \sup_{0\leq t\leq T}\|f(\cdot,t)\|_{C^{2,\gamma}(\rr)\cap{\dot{H}^k(\rr)}}.
\end{equation}
We need to prove that $S$ still maps $E$ to $E$ and that $S$ is contractive. The arguments are again similar. For instance, we can estimate the part of the homogeneous
derivatives as follows:
\begin{equation}\nonumber
\begin{split}
\sup_{0\leq t\leq T} \|\pa_x^k Sf\|_{L^2(\rr)} & \leq \|\pa_x^k f_0\|_{L^2}
                            + \sup_{t}\int_0^t \|\pa_x K^\epp(t-s)\ast\pa_x^{k-1}F(f(s))\|_{L^2}ds  \\[5pt]
                                               & \leq M + \sqrt{T}C(\epp) \sup_{t} \|\pa_x^{k-1}F(f(s))\|_{L^2}.
\end{split}
\end{equation}
To treat the right-hand side of the above inequality, we use \eqref{FFF} and \eqref{LinearPart} to derive that
\begin{equation}\nonumber
\begin{split}
 \sup_{t} \|\pa_x^{k-1}F(f(s))\|_{L^2} \leq & \sup_{t} \|\Lambda^{1-\epp}\pa_x^{k-1}f\|_{L^2} \\[5pt]
                    &+ \sup_{t} \int_{\rr}\|\pa_x^{k-1}\big{(}\frac{\pa_x \Delta_\alpha^\epp f(x) (\Delta_\alpha^\epp f(x))^2}
                    {1+(\Delta_\alpha^\epp f(x))^2}\big{)}\|_{L^2} d\alpha, \\[5pt]
                                  \leq& CM + S_1 + \text{lower order terms}, 
\end{split}
\end{equation}
where ~$S_1$~ denotes the most singular part. In the above inequality, we have also applied the following estimate
\begin{equation}\nonumber
\begin{split}
 \|\Lambda^{1-\epp}\pa_x^{k-1}f\|_{L^2(\rr)} \leq & \sum_j 2^{j(k-\epp)} \|f_j\|_{L^2(\rr)}   \\[5pt]
                                             \leq & \sum_{j\leq N_0}2^{j(k-\epp)} \|f_j\|_{L^2} + \sum_{j\geq N_0} 2^{j(k-\epp)}\|f_j\|_{L^2} \\[5pt]
                                             \leq &  \sum_{j\leq N_0} 2^{j(k-\epp-\frac{1}{2})} \|f_j\|_{L^\infty}+ \sum_{j\geq N_0} 2^{-j\epp}\||\pa|^k f_j\|_{L^2} \\[5pt]
                                             \leq &  C(\epp)(\|f\|_{L^\infty}+ \|f\|_{\dot{H}^k}).
\end{split}
\end{equation}
Here we have used the standard frequency cut-offs and Bernstein's inequality \cite{Chemin},
where $f_j$ represents the Fourier localization of $f$ and $\widehat{f_j}(\xi)$ is supported on the annulus with the inner and outer radius both around $2^j$.

We can estimate $S_1$ as follows:
\begin{equation}\nonumber
\begin{split}
S_1   &=  \sup_{t} \int_{\rr}\|\frac{\pa_x^k \Delta_\alpha^\epp f(x) (\Delta_\alpha^\epp f(x))^2}{1+(\Delta_\alpha^\epp f(x))^2}\|_{L^2} d\alpha\\[5pt]
      & \leq C\sup_{t}\|\pa_x^k f\|_{L^2} \|f\|_{L^\infty}^2 \int_{|\alpha|>1} \frac{1}{|\alpha|^{3-3\epp}}d\alpha
 + C \sup_{t} \|\pa_x^k f\|_{L^2} \int_{|\alpha|<1}\frac{1}{|\alpha|^{1-\epp}}d\alpha  \\[5pt]
      &\leq C(\epp)M^{3}.
\end{split}
\end{equation}
This completes the proof of Theorem \ref{LS}.
\end{proof}

Based on Theorem \ref{LS}, we claim that the approximate solutions $f(x,~t)$ have the same spatial asymptotics as the initial data $\bar{f}_0(x)$ at infinity. In fact, integrate \eqref{LAS}
in time and take limits to find that
\begin{equation}\label{HDOTK}
\begin{split}
\limit{x\rightarrow \infty}&|f(x,t)-\bar{f}_0(x)|   \\
&\leq \epp \limit{x\rightarrow \infty}\int_0^t |f_{xx}(x,s)|ds +
 \limit{x\rightarrow \infty}\int_0^t\int_{\rr}|\frac{\pa_x \Delta_\alpha^\epp f(x,s)}{1+(\Delta_\alpha^\epp f(x,s))^2}|d\alpha ds. 
\end{split}
\end{equation}
The righthand-side terms in \eqref{HDOTK} involve only the derivatives of $f(x,t)$, which
vanish at infinity since $f(x,t)$ belong to $ C^{2,\gamma}(\rr)\cap\dot{H}^1(\rr)\cap\dot{H}^3(\rr)$.
Exchanging the limit and the integration (due to the uniform convergence of the integral), one can find that the right-hand side goes to 0.

\section{Uniform \textit{a Priori} Estimates of the Approximate Solutions}

In this section we derive uniform estimates for the approximate solutions obtained in Section 3. More precisely, we have the following Lemma

\begin{lem}\label{3.2}
Let $f^\epp(x,t)$ be a regular solution of the system \eqref{LAS} with initial condition
\eqref{LIDA} \eqref{LIDB}. Then there is a positive $T$ and for any $0\leq t\leq T$, 
\begin{equation}\nonumber
\|f^\epp\|_{C^{2,\gamma}(\rr)\cap\dot{H}^1(\rr)\cap \dot{H}^3(\rr)}(t) \les \|\bar{f}_0\|_{C^{2,\gamma}(\rr)\cap\dot{H}^1(\rr)\cap \dot{H}^3(\rr)}. \quad\quad\quad  
\end{equation}
\end{lem}

\begin{proof} \quad As before, we drop the superscript $\epp$ in $f^\epp(x,t)$. The basic strategy of the proof is as follows. We first
estimate the $L^\infty$ norm of solutions of the approximate system \eqref{LAS}. To get
estimates in H\"{o}lder spaces,  we will apply the energy method and the classical Sobolev
imbeddings. Since the (artificial) viscosity term and the regularization of the
nonlinearity do no effect various estimates, we can get uniform in $\epp$ estimates and which can easily pass to limits to get desired estimates for the original system \eqref{OS}.

\textit{Maximum Principle} \quad Here we show the maximum value of $f(x,t)$ is
bounded. A similar argument holds for the minimum value. There are two possibilities:

\rm{(i)} $\bar{f}_0(x)$ attains its maximum value which is larger than $a$ in a finite
x-interval \\
$[-N,N]$. For this case, we follow the arguments in \cite{Maximum}. Let
\begin{equation}\nonumber
M(t)=\max_{x}f(x,t). 
\end{equation}
Then $M(t)$ is differentiable for almost every $t$. Suppose that ~$f(x,t)$ attain its
maximum value which is larger than $a$, then, by the spatial asymptotics \eqref{LIDA},
there is $x_t ~(x_t\neq \pm\infty)$ such that
\begin{equation}\nonumber
f(x_t,t)=\max_{x}f(x,t).
\end{equation}
It follows that
\begin{equation}\nonumber
f_x(x_t,t)=0,   \qquad  f_{xx}(x_t,t)\leq 0.
\end{equation}
As in \cite{Maximum}, we have (for almost all $t$)
\begin{equation}\nonumber
M\prim(t)=f_t(x_t,t).
\end{equation}

Now we calculate the derivative of $M(t)$
\begin{equation}\nonumber
\begin{split}
M\prim(t) & =\epp f_{xx}(x)|_{x=x_t}+\pa_x\int_{\rr} \arctan{(\Delta_\alpha^\epp f(x,t))} d\alpha |_{x=x_t} \\[5pt]
          & =I_1 + I_2 .
\end{split}
\end{equation}
Obviously one has $I_1 \leq 0$ and also
\begin{equation}\nonumber
\begin{split}
I_2=&\int_\rr \frac{f_x(x_t)}{\frac{x_t-\alpha}{|x_t-\alpha|^{\epp}}}\frac{1}{1+(\Delta_\alpha^\epp f(x_t,t))^2} d\alpha \\[5pt]
    &- (1-\epp) \int_\rr \frac{f(x_t)-f(\alpha)}{|x_t-\alpha|^{2-\epp}}\frac{1}{1+(\Delta_\alpha^\epp f(x_t,t))^2} d\alpha. 
\end{split}
\end{equation}
Since $x_t$ is a maximum point, one gets $f_x(x_t,t)=0$ and $f(x_t)-f(\alpha)\geq 0$.
Consequently, for $\epp$ smaller than 1, one has also that $I_2 \leq 0$. Thus $M\prim(t) \leq 0$ for almost every $t\in [0,T]$, which gives $M(t)\leq M(0)$.

\rm{(ii)}\ $\bar{f}_0(x)$ attains its maximum value $a$ at infinity.
We want to show the $M(t)$ can not be larger than $a$. Indeed, if for some
$ 0 < t < T$, $M(t)$ is larger than $a$, then the arguments in case \rm{(i)}
applies, and we conclude that $M(t)$ is decreasing at $t$. This suffices to imply again $M(t)$ is not larger than $a$ as it is true for $t=0$.

Analogously, we can get $m(t)=\min_{x}f(x,t) \geq m(0)$. Therefore, we obtain the maximum principle for the $L^\infty$ norm
\begin{equation}\label{3.m}
\|f\|_{L^\infty(\rr)}(t) \leq \|\bar{f}_0\|_{L^\infty(\rr)}. \\[10pt]
\end{equation}

\textit{Estimation on ~$\|\pa_x f(\cdot,t)\|_{L^2(\rr)}$} \quad Let us
differentiate once the system \eqref{OS} then multiply it by $f_x$, and then do an integration to derive that
\begin{equation}\label{3.3}
\begin{split}
   \frac{d}{dt}\|\pa_x f(\cdot,t)\|_{L^2}^2 + & \|\Lambda^{\frac{1}{2}}\pa_x f(\cdot,t)\|_{L^2}^2 \\[5pt]
  =& \int_{\rr} PV \int \pa_x( \frac{\pa_x \Delta_\alpha f(x,t)(\Delta_\alpha f(x,t))^2}{1+ (\Delta_\alpha f(x,t))^2} ) d\alpha \pa_x f(x,t)dx \\[5pt]
  =& K_1+K_2,
\end{split}
\end{equation}
where
\begin{equation}\nonumber
\begin{split}
 K_1 &=\int PV \int \frac{\pa_x^2 \Delta_\alpha f(x,t)(\Delta_\alpha f(x,t))^2}{1+ (\Delta_\alpha f(x,t))^2} d\alpha \pa_x f(x,t) dx, \\[5pt]
 K_2 &=2\int PV \int \frac{(\pa_x \Delta_\alpha f(x,t))^2 \Delta_\alpha f(x,t)}{(1+ (\Delta_\alpha f(x,t))^2)^2} d\alpha \pa_x f(x,t) dx. 
\end{split}
\end{equation}
$K_1$ can be estimated as follows:
\begin{equation}\nonumber
\begin{split}
|K_1|  \les & \int_{|\alpha|<1} \frac{1}{|\alpha|^{1-\gamma}}|\pa_x^2 f|_{\dot{C}^\gamma} \|\pa_x f\|_{L^\infty} \|\pa_x f\|_{L^2} \|\pa_x f\|_{L^2} d\alpha  \\[5pt]
          & + \int_{|\alpha|>1} \frac{1}{|\alpha|^2} \|\pa_x^2 f\|_{L^\infty} \|f\|_{L^\infty} \|\pa_x f\|_{L^2}^2 d\alpha  \\[5pt]
     \les & C(\|f\|_{L^\infty},|\pa_x^2 f|_{\dot{C}^\gamma} ) \|\pa_x f\|_{L^2}^2, 
\end{split}
\end{equation}
where we have used
\begin{equation}\nonumber
\begin{split}
\|\frac{f(\cdot)-f(\cdot-\alpha)}{\alpha}\|_{L_x^p} = \|\int_0^1 \pa_x f(\cdot-\alpha+ \theta\alpha)d\theta\|_{L^p_x} \les \|\pa_x f\|_{L_x^p}
\end{split}
\end{equation}
and
\begin{equation}\nonumber
\begin{split}
\|\frac{\pa_x^2f(\cdot)-\pa_x^2f(\cdot-\alpha)}{\alpha}\|_{L^\infty} \les \frac{1}{|\alpha|^{1-\gamma}} |\pa_x^2 f|_{\dot{C}^{\gamma}}.
\end{split}
\end{equation}
Similarly, we can deduce
\begin{equation}\nonumber
\begin{split}
|K_2| & \les \int_{|\alpha|<1} \|\pa_x^2 f\|_{L^\infty}^2 \|\pa_x f\|_{L^2}^2 d\alpha
+ \int_{|\alpha|>1} \frac{1}{\alpha^2} \|\pa_x f\|_{L^\infty}^2 \|\pa_x f\|_{L^2}^2 d\alpha  \\[5pt]
    & \les C(\|f\|_{L^\infty}, |\pa_x^2 f|_{\dot{C}^{\gamma}}) \|\pa_x f\|_{L^2}^2.
\end{split}
\end{equation}
Combining the estimates of $K_1$ and $K_2$ with \eqref{3.3}, we obtain
\begin{equation}\label{3.4}
\frac{d}{dt}\|\pa_x f(\cdot,t)\|_{L^2}^2 + \|\Lambda^{\frac{1}{2}}\pa_x f(\cdot,t)\|_{L^2}^2
  \les C(\|f\|_{C^{2,\gamma}}) \|\pa_x f\|_{L^2}^2. \\[10pt]
\end{equation}

\textit{Estimation on $\|\pa_x^3 f(\cdot,t)\|_{L^2(\rr)}$}\quad Again, the
standard energy method apply to the higher order derivatives yields
\begin{equation}\label{3.5}
\begin{split}
    \frac{d}{dt}\|\pa_x^3 f(\cdot,t)\|_{L^2}^2 + & \|\Lambda^{\frac{1}{2}}\pa_x^3 f(\cdot,t)\|_{L^2}^2   \\[5pt]
   =&\int_{\rr} PV \int \pa_x^3( \frac{\pa_x \Delta_\alpha f(x,t)(\Delta_\alpha f(x,t))^2}{1+ (\Delta_\alpha f(x,t))^2} ) d\alpha \pa_x^3 f(x,t)dx \\[5pt]
   =&K_3+K_4+K_5+K_6,
\end{split}
\end{equation}
where
\begin{equation}\nonumber
\begin{split}
K_3&= \int PV \int \frac{\pa_x^4 \Delta_\alpha f (\Delta_\alpha f)^2}{1+(\Delta_\alpha f)^2} d\alpha \pa_x^3 f(x) dx ,\\[5pt]
K_4&= \int PV \int \pa_x^3 \Delta_{\alpha}f \pa_x \big{(}\frac{(\Delta_\alpha f)^2}{1+(\Delta_\alpha f)^2} \big{)} d\alpha \pa_x^3 f(x) dx  ,\\[5pt]
K_5&= \int PV \int \pa_x^2 \Delta_{\alpha}f \pa_x^2 \big{(}\frac{(\Delta_\alpha f)^2}{1+(\Delta_\alpha f)^2} \big{)} d\alpha \pa_x^3 f(x) dx,  \\[5pt]
K_6&= \int PV \int \pa_x   \Delta_{\alpha}f \pa_x^3 \big{(}\frac{(\Delta_\alpha f)^2}{1+(\Delta_\alpha f)^2} \big{)} d\alpha \pa_x^3 f(x) dx.
\end{split}
\end{equation}
Let us first estimate $K_3$. Since the derivatives of $f(x,t)$ vanish at infinity, we use integration by parts to write
\begin{equation}\nonumber
\begin{split}
K_3  = & \int PV \int \frac{\pa_x \pa_x^3 f(x)-\pa_x \pa_x^3 f(x-\alpha)}{\alpha} \big{(}\frac{(\Delta_\alpha f)^2}{1+(\Delta_\alpha f)^2} \big{)}
\pa_x^3 f(x) d\alpha dx  \\[5pt]
     = & \frac{1}{2} \int PV \int \frac{1}{\alpha} \big{(}\frac{(\Delta_\alpha f)^2}{1+(\Delta_\alpha f)^2} \big{)} d\alpha \pa_x(\pa_x^3 f)^2
    dx  \\[5pt]
       & + \int PV \int \frac{\pa_\alpha \pa_x^3 f(x-\alpha)}{\alpha} \big{(}\frac{(\Delta_\alpha f)^2}{1+(\Delta_\alpha f)^2} \big{)} d\alpha
    \pa_x^3 f(x) dx \\[5pt]
     = & K_{31}+K_{32} .
\end{split}
\end{equation}
Here $K_{31}$ can be estimated as follows:
\begin{equation}\nonumber
\begin{split}
K_{31}&= -\frac{1}{2} \int PV \int \frac{1}{\alpha} \pa_x(1-\frac{1}{1+(\Delta_\alpha f)^2}) d\alpha (\pa_x^3 f)^3 dx  \\[5pt]
   &=- \int PV \int \frac{1}{\alpha} \frac{\Delta_\alpha f (\pa_x \Delta_\alpha f)}{(1+(\Delta_\alpha f)^2)^2} d\alpha (\pa_x^3 f)^2 dx \\[5pt]
   &=- \int (\pa_x^3 f)^2 PV \int_{|\alpha|>1}d\alpha dx - \int (\pa_x^3 f)^2 PV \int_{|\alpha|<1} d\alpha dx \\[5pt]
   &\les \|\pa_x f\|_{L^\infty}^2 \|\pa_x^3 f\|_{L^2}^2 \int_{|\alpha|>1} \frac{1}{|\alpha|^2}d\alpha
   + M(f) \|\pa_x^3 f\|_{L^2}^2,
\end{split}
\end{equation}
where
\begin{equation}\nonumber
M(f)= \sup_{x} |PV \int_{|\alpha|<1}\frac{1}{\alpha} \frac{(\frac{f(x)-f(x-\alpha)}{\alpha})(\frac{\pa_x f(x)-\pa_x f(x-\alpha)}{\alpha})}{(1+(\frac{f(x)-f(x-\alpha)}{\alpha})^2)^2} d\alpha |. 
\end{equation}
We thus need to estimate $M(f)$. This can be done as follows:
\begin{equation}\nonumber
\begin{split}
M(f) \les & \sup_{x}|PV \int_{|\alpha|<1}\frac{1}{\alpha} \frac{(\frac{f(x)-f(x-\alpha)}{\alpha}-\pa_x f(x))(\frac{\pa_x f(x)-\pa_x f(x-\alpha)}{\alpha})}{(1+(\frac{f(x)-f(x-\alpha)}{\alpha})^2)^2} d\alpha  |   \\[5pt]
          & +\sup_{x} |PV \int_{|\alpha|<1}\frac{1}{\alpha} \frac{\pa_x f(x)(\frac{\pa_x f(x)-\pa_x f(x-\alpha)}{\alpha}-\pa_x^2 f(x))}{(1+(\frac{f(x)-f(x-\alpha)}{\alpha})^2)^2} d\alpha  |    \\[5pt]
          & +\sup_{x} |PV \int_{|\alpha|<1} \frac{1}{\alpha} \pa_x f(x)\pa_x^2 f(x) \big{(} \frac{1}{(1+(\Delta_\alpha f)^2)^2} -\frac{1}{(1+(\pa_x f)^2)^2} \big{)} d\alpha  |  \\[5pt]
          & + \sup_{x} |PV \int_{|\alpha|<1} \frac{1}{\alpha} \frac{\pa_x f(x) \pa_x^2 f(x)}{(1+ (\pa_x f)^2)^2} d\alpha |  \\[5pt]
\triangleq & K_{311} + K_{312} + K_{313} + K_{314}.
\end{split}
\end{equation}
Note that $K_{314}$ vanishes due to principle value integral. To estimate $K_{311}$, we observe the following formula:
\begin{equation}\nonumber
\begin{split}
   &\sup_{x}|\frac{1}{\alpha}(\frac{f(x)-f(x-\alpha)}{\alpha}-\pa_x f(x))| \\[5pt]
      = & \sup_{x}|\frac{1}{\alpha}\int_0^1 \pa_x f(x-\alpha+\theta\alpha)-\pa_x f(x)
 d\theta|  \\[5pt]
                              = &\sup_{x}|\frac{1}{\alpha^{1-\gamma}} \int_0^1 \frac{\pa_x f(x-\alpha+\theta\alpha)-\pa_x f(x)}{((\theta-1)\alpha)^\gamma}(\theta-1)^\gamma d\theta| \\[5pt]
                               \les  &\frac{1}{|\alpha|^{1-\gamma}} |\pa_x f|_{\dot{C}^{\gamma}}. 
\end{split}
\end{equation}
Consequently, one has
\begin{equation}\nonumber
|K_{311}| \les \|\pa_x^2 f\|_{L^\infty} |\pa_x f|_{\dot{C}^\gamma}.
\end{equation}
By similar reasoning, one also concludes
\begin{equation}\nonumber
\begin{split}
|K_{312}| &\les \|\pa_x f\|_{L^\infty}|\pa_x^2 f|_{\dot{C}^\gamma}, \\[5pt]
|K_{313}| &\les \|\pa_x^2 f\|_{L^\infty} \|\pa_x f \|_{L^\infty}^2  (1+\|\pa_x f\|_{L^\infty}^2 )|\pa_x f|_{\dot{C}^\gamma}, 
\end{split}
\end{equation}
which yields the estimate of $K_{31}$:
\begin{equation}\nonumber
|K_{31}| \les C(\|f\|_{L^\infty},|\pa_x^2 f|_{\dot{C}^\gamma}) \|\pa_x^3 f\|_{L^2}^2.
\end{equation}
To estimate $K_{32}$, one again use the integration by parts to see
\begin{equation}\nonumber
\begin{split}
K_{32} = & -\int PV \int \frac{\pa_x^3 f(x)-\pa_x^3 f(x-\alpha)}{\alpha^2} \frac{(\Delta_\alpha f)^2}{1+(\Delta_\alpha f)^2}  d\alpha \pa_x^3 f(x) dx  \\[5pt]
      & + 2\int PV \int \frac{\pa_x^3 f(x)-\pa_x^3 f(x-\alpha)}{\alpha} \frac{\Delta_\alpha f \pa_\alpha \Delta_\alpha f}{(1+ (\Delta_\alpha f)^2)^2} d\alpha \pa_x^3 f(x) dx  \\[5pt]
    = & K_{321}+K_{322}.
\end{split}
\end{equation}
By changing variables in $K_{321}$, one obtains
\begin{equation}\nonumber
\begin{split}
K_{321} =& -PV \int \int \frac{\pa_x^3 f(x)-\pa_x^3 f(\alpha)}{(x-\alpha)^2} \frac{(\Delta_{x-\alpha} f)^2}{1+(\Delta_{x-\alpha} f)^2} \pa_x^3 f(x) d\alpha dx \\[5pt]
    =& PV \int \int \frac{\pa_x^3 f(x)-\pa_x^3 f(\alpha)}{(x-\alpha)^2} \frac{(\Delta_{x-\alpha} f)^2}{1+(\Delta_{x-\alpha} f)^2} \pa_x^3 f(\alpha) d\alpha dx \\[5pt]
    =& -\frac{1}{2}PV \int \int \frac{(\pa_x^3 f(x)-\pa_x^3 f(\alpha))^2}{(x-\alpha)^2} \frac{(\Delta_{x-\alpha} f)^2}{1+(\Delta_{x-\alpha} f)^2} d\alpha dx   \\[5pt]
    \leq & 0.
\end{split}
\end{equation}
For $K_{322}$, one proceeds as follows:
\begin{equation}\nonumber
\begin{split}
K_{322} =& 2\int |\pa_x^3 f(x)|^2 PV \int \frac{1}{\alpha} \frac{\pa_x f(x-\alpha)-\Delta_\alpha f}{\alpha}  \frac{\Delta_\alpha f}{(1+(\Delta_\alpha f)^2)^2}d\alpha dx \\[5pt]
     & +2 \int \pa_x^3 f(x) PV \int \pa_x^3 f(\alpha) \frac{f(x)-f(\alpha)}{(x-\alpha)^2} \\[5pt]
     &   \qquad \qquad \qquad \qquad \times \frac{\frac{f(x)-f(\alpha)}{x-\alpha}-\pa_x f(\alpha)}{x-\alpha}
      \frac{1}{(1+(\Delta_{x-\alpha} f)^2)^2}d\alpha dx   \\[5pt]
    =& K_{3221} + K_{3222}.
\end{split}
\end{equation}
Here $K_{3221}$ can be estimated similarly as for $K_{31}$. For $K_{3222}$, we calculate further to get
\begin{equation}\nonumber
\begin{split}
K_{3222} = & 2PV \int\int \pa_x^3 f(x) \pa_x^3 f(\alpha) \frac{f(x)-f(\alpha)}{(x-\alpha)^2} \frac{\frac{1}{4}\pa_x^2 f(\alpha)+\frac{1}{4} \pa_x^2 f(x)}{(1+(\Delta_{x-\alpha}f)^2)^2}  d\alpha dx  \\[5pt]
      & +2 PV \int\int \frac{\pa_x^3 f(x) \pa_x^3 f(\alpha) }{(1+(\Delta_{x-\alpha}f)^2)^2} \frac{f(x)-f(\alpha)}{(x-\alpha)^2}  \\[5pt]
      &\ \ \ \ \ \ \ \ \qquad  \times \Big{(} \frac{\frac{f(x)-f(\alpha)}{(x-\alpha)}-\pa_x f(\alpha)}{x-\alpha} -\frac{1}{4}\pa_x^2 f(x)-\frac{1}{4}\pa_x^2 f(\alpha)   \Big{)}d\alpha dx \\[5pt]
    = & K_{3222}^1 + K_{3222}^2 .
\end{split}
\end{equation}
By exchanging variables again, one finds that $K_{3222}^1=0$. The term $K_{3222}^2$ can be estimated as follows:
\begin{equation}\nonumber
\begin{split}
|K_{3222}^2| \les & \int |\pa_x^3 f(x)|^2 \int \frac{|f(x)-f(x-\alpha)|}{|\alpha|^2}  \\[5pt]
         &\ \ \ \times \big{(} \frac{\frac{f(x)-f(x-\alpha)}{\alpha}-\pa_x f(x-\alpha)}{\alpha}-\frac{1}{4}\pa_x^2 f(x)-\frac{1}{4}\pa_x^2 f(x-\alpha)    \big{)} d\alpha dx \\[5pt]
         & + \int |\pa_x^3 f(\alpha)|^2 \int \frac{|f(x+\alpha)-f(\alpha)|}{|x|^2}   \\[5pt]
         &\ \ \ \ \ \ \times \big{(} \frac{\frac{f(x+\alpha)-f(\alpha)}{x}-\pa_x f(\alpha)}{x}-\frac{1}{4}\pa_x^2 f(x+\alpha)-\frac{1}{4}\pa_x^2 f(\alpha)    \big{)}  dx d\alpha \\[5pt]
    \les & \int |\pa_x^3 f(x)|^2 (\int_{|\alpha|<1}d\alpha+\int_{|\alpha|>1}d\alpha)dx \\[5pt]
         &+ \int |\pa_x^3 f(\alpha)|^2
    (\int_{|x|<1}dx+\int_{|x|>1} dx) d\alpha \\[5pt]
    \les & C(\|f\|_{L^\infty}, |\pa_x^2 f|_{\dot{C}^{\gamma}}) \|\pa_x^3 f\|_{L^2}^2 .
\end{split}
\end{equation}
Collecting all these estimates related to $K_3$, we finally arrive at
\begin{equation} \label{3.6}
|K_3|\les C( \|f\|_{C^{2,\gamma}}) \|\pa_x^3 f\|_{L^2}^2. 
\end{equation}
For $K_4$, we can repeat much of the above with similar calculations:
\begin{equation}\nonumber
\begin{split}
K_4 =& 2PV \int |\pa_x^3 f(x)|^2 \int \frac{1}{\alpha} \frac{\Delta_\alpha f \pa_x \Delta_\alpha f}{(1+(\Delta_\alpha f)^2)^2} d\alpha dx \\[5pt]
     & -2 PV\int \pa_x^3 f(x) \int \frac{\pa_x^3 f(\alpha)}{x-\alpha} \frac{\Delta_{x-\alpha}f  \pa_x
     \Delta_{x-\alpha}f}{(1+(\Delta_{x-\alpha}f)^2)^2} d\alpha dx \\[5pt]
    =& K_{41} + K_{42}.
\end{split}
\end{equation}
$K_{41}$ can be estimated similarly as $K_{31}$:
\begin{equation}\nonumber
|K_{41}| \les C(\|f\|_{L^\infty}, |\pa_x^2 f|_{\dot{C}^\gamma}) \|\pa_x^3 f\|_{L^2}^2.
\end{equation}
We decompose $K_{42}$ into four terms:  $K_{42} = K_{421}+ K_{422} + K_{423} + K_{424}$, where
\begin{equation}\nonumber
\begin{split}
 K_{421}=& -2PV \int\int \pa_x^3 f(x) \pa_x^3 f(\alpha) \frac{\Delta_{x-\alpha}f-\pa_x f(\alpha)}{x-\alpha} \frac{\pa_x\Delta_{x-\alpha}f}{(1+(\Delta_{x-\alpha}f)^2)^2} d\alpha dx, \\[5pt]
 K_{422}=& -2PV \int\int \pa_x^3 f(x) \pa_x^3 f(\alpha) \pa_x f(\alpha)\\[5pt]
  &\qquad\qquad\quad\times\frac{\pa_x \Delta_{x-\alpha}f -\pa_x^2 f(\alpha)}{x-\alpha} \frac{1}{(1+(\Delta_{x-\alpha}f)^2)^2} d\alpha dx, \\[5pt]
\end{split}
\end{equation}

\begin{equation}\nonumber
\begin{split}
 K_{423} =& -2PV \int\int \pa_x^3 f(x) \pa_x^3 f(\alpha) \pa_x f(\alpha) \pa_x^2 f(\alpha) \\[5pt]
     & \ \ \ \ \ \qquad \qquad \times \frac{1}{x-\alpha} \Big{(} \frac{1}{(1+(\Delta_{x-\alpha}f)^2)^2}-\frac{1}{(1+(\pa_x f(\alpha))^2)^2} \Big{)} d\alpha dx , \\[5pt]
 K_{424}=& \quad 2\int PV\int \frac{\pa_x^3 f(x)}{\alpha-x}dx   \frac{\pa_x^3f(\alpha)\pa_x f(\alpha) \pa_x^2 f(\alpha)}{1+(\pa_x f(\alpha))^2} d\alpha .
\end{split}
\end{equation}
$K_{421}, K_{422}, K_{423}$ can be estimated similarly to $K_{311},K_{312},K_{313}$. For $K_{424}$, due to fact that the Hilbert transform is bounded in $L^p ~ (1<p<\infty)$ \cite{Stein}, one has
\begin{equation}\nonumber
\begin{split}
|K_{424}| \les & \|PV\int \frac{\pa_x^3 f(x)}{\alpha-x}dx \|_{L^2_\alpha} \|\frac{\pa_x^3f(\alpha)\pa_x f(\alpha) \pa_x^2 f(\alpha)}{1+(\pa_x f(\alpha))^2}\|_{L^2_\alpha} \\[5pt]
    \les & C(\|\pa_x f\|_{L^\infty}, \|\pa_x^2 f\|_{L^\infty}) \|\pa_x^3 f\|_{L^2}^2. 
\end{split}
\end{equation}
Thus one gets
\begin{equation}\label{3.7}
|K_4| \les  C(\|f\|_{L^\infty},|\pa_x^2 f|_{\dot{C}^\gamma}) \|\pa_x^3 f\|_{L^2}^2 .
\end{equation}
We can estimate $K_5$ directly as follows:
\begin{equation}\label{3.8}
\begin{split}
|K_5| \les & \int |\pa_x^3 f(x)| \int |\frac{\pa_x^2 f(x)-\pa_x^2 f(x-\alpha)}{\alpha}|  \\[5pt]
           &\ \ \ \ \times(|\pa_x \Delta_\alpha f|^2 + |\Delta_\alpha f \pa_x^2 \Delta_\alpha f|+ |\Delta_\alpha f|^2 |\pa_x \Delta_\alpha f|^2) d\alpha dx \\[5pt]
      \les & \|\pa_x^3 f\|_{L^2}^2 (\int_{|\alpha|<1} d\alpha + \int_{|\alpha|>1} d\alpha)  \\[5pt]
      \les & C(\|f\|_{L^\infty}, |\pa_x^2 f|_{\dot{C}^\gamma})\|\pa_x^3 f\|_{L^2}^2.  
\end{split}
\end{equation}
For $K_6$, the most troublesome term (after natural decompositions as we did above), which is also the
most singular part in $K_6$,  will be denoted by $K_{61}$. It is defined as follows:
\begin{equation}\nonumber
\begin{split}
K_{61} = & 2\int \pa_x^3 f(x) PV \int \frac{\Delta_\alpha f \pa_x \Delta_\alpha f \pa_x^3 \Delta_\alpha f}{1+(\Delta_\alpha f)^2}d\alpha dx  \\[5pt]
    = & 2 \int \pa_x^3 f(x) PV \int \frac{\pa_x^3 f(x)- \pa_x^3 f(\alpha)}{x-\alpha}
     \frac{\Delta_{x-\alpha} f \pa_x \Delta_{x-\alpha} f}{1+(\Delta_{x-\alpha}f)^2} d\alpha dx  \\[5pt]
    = & 2 \int |\pa_x^3 f(x)|^2 PV \int \frac{1}{\alpha} \frac{\Delta_\alpha f \pa_x\Delta_\alpha f}{1+(\Delta_\alpha f)^2} d\alpha dx \\[5pt]
      & -2 PV \int\int \frac{\pa_x^3 f(x) \pa_x^3f(\alpha)}{x-\alpha} \frac{\Delta_\alpha f \pa_x\Delta_\alpha f}{1+(\Delta_\alpha f)^2} d\alpha dx \\[5pt]
    = & K_{611}+ K_{612} ,
\end{split}
\end{equation}
where $K_{611}$ and $K_{612}$ can again be estimated similarly as $K_{41}$ and $K_{42}$. Therefore we have
\begin{equation}
\begin{split}
|K_{61}| \les  |K_{611}| + |K_{612}|  \les  C(\|f\|_{L^\infty}, |\pa_x^2 f|_{\dot{C}^\gamma}) \|\pa_x^3 f\|_{L^2}^2.
\end{split}
\end{equation}
Consequently, one concludes
\begin{equation}\label{3.9}
|K_6| \les C(\|f\|_{C^{2,\gamma}}) \|\pa_x^3 f\|_{L^2}^2.
\end{equation}
Combining \eqref{3.6}, \eqref{3.7}, \eqref{3.8} and \eqref{3.9}, we then obtain
\begin{equation}\label{3.10}
\begin{split}
\frac{d}{dt} \|\pa_x^3 f\|_{L^2}^2 + \|\Lambda^{\frac{1}{2}}\pa_x^3 f\|_{L^2}^2 \les & C(\|f\|_{C^{2,\gamma}}) \|\pa_x^3 f\|_{L^2}^2  \\[5pt]
                                                                                \les & C(\|f\|_{L^\infty}) \|\pa_x^3 f\|_{L^2}^{k_0},
\end{split}
\end{equation}
where we have used the classical Sobolev embedding from Sobolev spaces to H\"{o}lder Spaces.
Combining \eqref{3.m}, \eqref{3.4} and \eqref{3.10}, we conclude that, there exist a time $T$,  such that $\|f(\cdot,t)\|_{C^{2,\gamma}(\rr)\cap\dot{H}^1(\rr)\cap \dot{H}^3(\rr)}
(0<\gamma<\frac{1}{2})$ is uniformly bounded for all $t\in [0,T]$. 
\end{proof}

\section{Convergence of Local solutions in H\"{o}lder Spaces}

In this section, we are going to end the proof of Theorem \ref{LW} by establishing some
compactness results. From the \textit{a priori} estimates in Section 4, we have the
following uniform bounds in $\epp$
\begin{equation}\label{4.1}
\sup_{0\leq t \leq T}\|f^{\epp}\|_{C^{2,\gamma}(\rr)\cap\dot{H}^1(\rr)\cap \dot{H}^3(\rr)} \leq C(T), \quad  \sup_{0\leq t \leq T}\|\pa_t f^\epp\|_{C^{\gamma}(\rr)}\leq C(T),
\end{equation}
for $0<\gamma<\frac{1}{2}.$ One uses system \eqref{OS} to derive the second bound in \eqref{4.1}. In order to pass
to the limits $\epp \rightarrow 0$ for the approximate system \eqref{LAS}, we need
the following two elementary facts:

\begin{lem}\label{4.2}
Let $ u_k(x) $ be any bounded sequence in $C^{2,\gamma}(\rr)\cap\dot{H}^1(\rr)\cap
\dot{H}^3(\rr)(0<\gamma<\frac{1}{2})$ with the uniform spatial asymptotics \eqref{LIDA}. Then $u_k(x)$ is strongly convergent in ~$C^{2,\gamma} (\rr)$ up to a subsequence.
\end{lem}

\begin{lem}\label{4.3}
Let $ u_k(x,t)$ be any uniformly bounded sequence satisfying \eqref{4.1} and with the uniform spatial asymptotics \eqref{LIDA}. Then $u_k(x,t)$ is strongly convergent in $L^\infty([0,T];C^{2,\gamma}(\rr))(0<\gamma<\frac{1}{2})$ up to a subsequence.
\end{lem}

\begin{rem}
Theorem \ref{LW} is obviously a direct consequence of Lemma \ref{4.3} by taking $\epp\rightarrow 0$ in the approximate system \eqref{LAS}.
\end{rem}

\begin{proof} \textit{Lemma 5.1}\quad It suffices to prove that $ u_k(x)$ is compact in $C^{\gamma} (\rr)$, since the compactness in $C^{2,\gamma} (\rr)$ follows from a standard interpolation argument. Let $\chi_R(x)(R>0)$ be the standard smooth cut-off function such that
\begin{equation}\nonumber
\chi_R(x)=
\begin{cases}
1, \quad\quad x\in B_R, \\
0, \quad\quad x\in B_{2R}^c.
\end{cases}
\end{equation}
For any integers $k_1, k_2$, one has
\begin{equation} \label{4.4}
\begin{split}
\|u_{k_1}(x)-u_{k_2}(x)\|_{C^\gamma} \leq & \| |u_{k_1}(x)-u_{k_2}(x)|{\chi_{R}} + |u_{k_1}(x)-u_{k_2}(x)|(1-\chi_R) \|_{C^\gamma}  \\[5pt]
                                     \leq & \| u_{k_1}-u_{k_2} \|_{C^\gamma(B_{2R})} + \| u_{k_1}-u_{k_2} \|_{C^\gamma(B_R^c)}. 
\end{split}
\end{equation}
Lemma \ref{4.2} is proved if we show that the two terms in the last line of \eqref{4.4} can be arbitrarily small provided $R$, $k_1$ and $k_2$ are large enough. For the second term in \eqref{4.4}, we use the Sobolev imbedding theorem to derive that
\begin{equation}\label{4.5}
\begin{split}
\| u_{k_1}&-u_{k_2} \|_{C^\gamma(B_R^c)} \\[5pt]
                                        & \leq \| u_{k_1}-u_{k_2} \|_{L^\infty(B_R^c)}  +\sup_{0<|x-y|<1} \frac{|(u_{k_1}-u_{k_2})(x)-(u_{k_1}-u_{k_2})(y)|}{|x-y|^\gamma}   \\[5pt]
                                        & \leq \| u_{k_1}-u_{k_2} \|_{L^\infty(B_R^c)}  +  \sup_{0<|x-y|<1} \frac{|(u_{k_1}-u_{k_2})(x)-(u_{k_1}-u_{k_2})(y)|}{|x-y|}   \\[5pt]
                                         & \leq \| u_{k_1}-u_{k_2} \|_{L^\infty(B_R^c)}+ \|\pa_x(u_{k_1}-u_{k_2})\|_{L^\infty(B_R^c)} \\[5pt]
                                         & \leq \| u_{k_1}-u_{k_2} \|_{L^\infty(B_R^c)}+ C \|\pa_x(u_{k_1}-u_{k_2})\|_{H^1(B_R^c)}. 
\end{split}
\end{equation}
Noting that $ u_k(x)$ have the asymptotics \eqref{LIDA}, and that ~$ u_k(x) $ is uniformly bounded in $C^{2,\gamma}(\rr)\cap\dot{H}^1(\rr)\cap \dot{H}^3(\rr)$, we know that the two terms in the last line of \eqref{4.5} tends to 0 after taking $R$, $k_1$ and $k_2$ sufficiently large. By the compact imbedding from $C^{2,\gamma}(B_{2R})$ into $C^\gamma(B_{2R})$, the first term in \eqref{4.4} also converges to 0 up to a subsequence. 
\end{proof}

\begin{proof} \textit{Lemma 5.2} \quad As in the proof of Lemma \ref{4.2}, it suffices to prove the compactness in $L^\infty([0,T],C^\gamma(\rr))$. $\forall ~ \ep>0, ~\forall~t_1,t_2 \in [0,T]$, one has
\begin{equation}\nonumber
\begin{split}
\|u_k(x,t_1)-u_k(x,t_2)\|_{C^{\gamma}(\rr)}  \leq & \sup_{0\leq t \leq T}\|\pa_t u_k\|_{C^\gamma(\rr)}|t_1-t_2|  \\[5pt]
                                               \leq & M|t_1-t_2|,
\end{split}
\end{equation}
here $M=\sup_k \sup_{0\leq t \leq T} \|\pa_t u_k\|_{C^\gamma(\rr)}$. Since $[0,T]$ is a compact set, there exists a finite $\frac{\ep}{3M}$-net $\{ t_i| 0\leq t_i \leq T, 1\leq i \leq m \}$ s.t. for any
$0\leq t \leq T$, there exists a $t_j$ s.t. $|t-t_j|<\frac{\ep}{3M}$.
From Lemma \ref{4.2} we know that for any fixed $t_j$, there exists a $N_j$ s.t. for any $k_1,k_2>N_j$, there holds
\begin{equation}\nonumber
\|u_{k_1}(\cdot,t_j)-u_{k_2}(\cdot,t_j)\|_{C^\gamma(\rr)} < \frac{\ep}{3}.
\end{equation}
Taking $N=\max\{N_1,N_2,...,N_j\}$ for any $k_1,k_2>N$, one has
\begin{equation}\nonumber
\sup_{t\in[0,T]}\|u_{k_1}(\cdot,t)-u_{k_2}(\cdot,t)\|_{C^\gamma(\rr)} < \ep,
\end{equation}
which ends the proof of Lemma \ref{4.3}. 
\end{proof}

The following remark will be used in the proof of Lemma \ref{5.2}:
\begin{rem}
Notice that $\beta x$ is a solution to the system \eqref{OS} for any constant $\beta$. Using this fact and similar arguments as in the proof of Theorem \ref{LW}, we can show that system \eqref{OS} is also locally well-posed for initial data of the form
\begin{equation}
  f^\beta(x,0)=\beta x+ f_0(x),
\end{equation}
where $f_0(x)$ is given in Theorem \ref{LW}. Correspondingly, the local solution is given in the form $f^\beta(x,t)=\beta x + \widetilde{f}^\beta(x,t)$, where $\widetilde{f}^\beta(x,t)$ satisfies
\begin{equation}\label{beta}
\begin{cases}
&\widetilde{f}^\beta_t(x,t)=\frac{\rho^2-\rho^1}{2\pi} PV\int_{\rr}\frac{\pa_x\Delta_\alpha \widetilde{f}^\beta(x,t)}{1+(\Delta_\alpha \widetilde{f}^\beta(x,t)+\beta)^2}d\alpha ,\\[5pt]
&\widetilde{f}^\beta(x,0)=f_0(x), \quad \quad x\in \rr,
\end{cases}
\end{equation}
and
\begin{equation}\label{beta1}
\widetilde{f}^\beta(x,t)\in C([0,T],C^{2,\gamma}(\rr)\cap\dot{H}^1(\rr)\cap \dot{H}^3(\rr)).
\end{equation}
Moreover, the norm of $\widetilde{f}^\beta(x,t)$ in $C([0,T],C^{2,\gamma}(\rr)\cap\dot{H}^1(\rr)\cap \dot{H}^3(\rr))$ is independent of $\beta$ for all $|\beta|< 1$. Here for simplicity, we can still denote $\widetilde{f}^\beta(x,t)$ by $f(x,t)$.
\end{rem}

\section{Global Weak Solutions}

In this section, we prove the existence of global weak solutions stated in
Theorem \ref{MAINTHM}. We call $f(x,t)$ a weak solution if it satisfies the system
\eqref{OS} in the sense of distribution
\begin{equation}\label{WS}
\begin{split}
\int_0^T\int_\rr f(x,t)&\phi_t(x,t)dx dt + \int_\rr f_0(x)\phi(x,0)dx  \\[5pt]
&= \int_0^T \int_\rr PV \int_\rr \arctan (\frac{f(x,t)-f(x-\alpha,t)}{\alpha})d\alpha \phi_x(x,t)dxdt, 
\end{split}
\end{equation}
for all $\phi(x,t)\in C_c^\infty([0,T)\times \rr)$.

Let us introduce the following regularized system as in \cite{Peter}:
\begin{equation}\label{5.1}
\begin{cases}
&f^\ep_t(x,t)=-\ep C \Lambda^{1-\ep}f^\ep + \ep f^\ep_{xx} +  PV \pa_x\int_\rr \arctan(\Delta_\alpha^\ep f^\ep)d\alpha, \\[5pt]
&f^\ep(x,t)=f^\ep_0(x), 
\end{cases}
\end{equation}
where $C>0$ is an universal constant which will be determined later. The operator $\Lambda^{1-\ep}f$ is given by the formula
\begin{equation}\label{5.3}
\begin{split}
\Lambda^{1-\ep}f(x)= & c_1(\ep) \int_\rr \frac{f(x)-f(x-\alpha)}{|\alpha|^{2-\ep}}d\alpha, 
\end{split}
\end{equation}
with $ A_1 \leq c_1(\ep) \leq A_2$.\ $A_1,A_2>0$ are two constants (independent of $\ep$) provided that $\ep$ is suitably small. Initial data for the regularized system \eqref{5.1} is given by
the mollification of \eqref{IDB}
\begin{equation}
f^\ep_0(x)=\JJ f_0(x),
\end{equation}
where $\JJ$ is a standard mollifier \cite{Majda}. Note that the mollified initial data $f^\ep_0(x)$ is monotonically decreasing since $f_0(x)$ is. Moreover, we have
\begin{equation}\label{3.1}
\begin{split}
&f^\ep_0(x)\in C^{\infty}(\rr)\cap \dot{H}^k(\rr),~~for~~any~~k\geq 1,  \\[5pt]
&f^\ep_0(x)\rightarrow a~~as~ x\rightarrow -\infty, \ \  f^\ep_0(x)\rightarrow b~~as~ x\rightarrow +\infty~~\ \ (b>a). 
\end{split}
\end{equation}
Again the (artificial) viscous term $-\ep C \Lambda^{1-\ep}f(x,t)$ in \eqref{5.1} will
cause no trouble in deriving those \textit{a priori} estimates in Section 4, and from the previous discussion in Section 3 and 4 we know that for any fixed $\ep$ the regularized system admits a unique local regular solution. We remark that it is proved in \cite{Peter} that the regularized system with $H^k(\rr) (k\geq 3)$ initial data admits global regular solutions. Although our initial data has infinite energy, the spatial asymptotics
\eqref{IDA} enable us to utilize higher order energy estimates. One then can also deduce an existence result for global regular solutions. We leave details to the readers.

\begin{lem}\label{5.2}
Let ~$f^\ep(x,t)$~ be a regular solution to the system \eqref{5.1}. Then for any ~$t>0$, there holds
\begin{equation}\nonumber
\|f^\ep\|_{L^\infty(\rr)}(t) \leq \|f_0\|_{L^\infty(\rr)}, \quad\quad \|\pa_x f^\ep\|_{L^\infty(\rr)}(t) \leq \|\pa_x f_0\|_{L^\infty(\rr)}.
\end{equation}
\end{lem}

\begin{proof}
We drop the superscript $\ep$ and prove
first the maximum principle for $\pa_x f(x,t)$. Here we only show the part of maximum
value since the part of minimum value can be obtained by a similar argument. From the
spatial asymptotics \eqref{IDA} we know that $\pa_x f(x,t)\rightarrow 0, ~as~
x\rightarrow \pm \infty$. Since the initial data $f_0(x)$ is monotonically decreasing, we need to show that $f(x,t)$ maintains the monotonic property, i.e. the maximum value of $f_x(x,t)$ cannot exceed 0. If $f_x(x,t)$ attains its maximum value at infinity for all $0\leq t \leq T$, there is nothing to prove. Thus one suffices to examine those time $t_0$ such that
there exist a point $x_{t_0}(x_{t_0}\neq \pm\infty)$ such that $f_x(x,t)$ attains its maximum value at $(x_{t_0},t_0)$. Moreover, $x=x_{t_0}$ is a stationary point.
Denote
\begin{equation}\nonumber
M(t)=\max_{x}f_x(x,t).
\end{equation}
Suppose that $f_x(x,t)$ attains its maximum value at $x_t(x_t\neq \pm\infty
)$, then it is clear that
\begin{equation}\nonumber
f_{xx}(x_t,t)=0, \quad\quad f_{xxx}(x_t,t)\leq 0.
\end{equation}
Moreover, for almost such $t$ that $\pa_x f$ reaches its maximum value at finite spatial point $x_t$, one has
\begin{equation*}
M\prim(t)=f_{xt}(x_t,t).
\end{equation*}
The evolution of $M\prim(t)$ is govern by
\begin{equation}\nonumber
M\prim(t)=f_{xt}(x_t,t)=-\ep C \Lambda^{1-\ep}f_x(x_t)+\ep f_{xxx}(x_t) + I_x(x_t),
\end{equation}
where $I(x)$ is given by
\begin{equation}\nonumber
\begin{split}
I(x)= & \pa_x f(x) PV \int_\rr \frac{\frac{1}{\phi(x-\alpha)}}{1+(\Delta_{x-\alpha}^\ep f(x) )^2} d\alpha
      - (1-\ep) PV \int_\rr \frac{\frac{f(x)-f(\alpha)}{|x-\alpha|^{2-\ep}}}{1+(\Delta_{x-\alpha}^\ep f(x) )^2} d\alpha  \\[5pt]
    = & PV \int_\rr \frac{\frac{f_x(x)-\frac{f(x)-f(\alpha)}{x-\alpha}}{\phi(x-\alpha)}}{1+(\Delta_{x-\alpha}^\ep f(x) )^2} d\alpha
      + \ep PV \int_\rr \frac{\frac{f(x)-f(\alpha)}{|x-\alpha|^{2-\ep}}}{1+(\Delta_{x-\alpha}^\ep f(x) )^2} d\alpha  \\[5pt]
    \triangleq & I_1+ I_2. \qquad(~\phi(\alpha)\triangleq\frac{\alpha}{|\alpha|^\ep}~)
\end{split}
\end{equation}
Then one has
\begin{equation}\nonumber
I_x(x)=\pa_xI_1(x) + \pa_x I_2(x),
\end{equation}
where
\begin{equation}\nonumber
\begin{split}
\pa_x I_1(x)= & PV \int \frac{\frac{f_{xx}(x)}{\phi(x-\alpha)}}{1+(\Delta_{x-\alpha}^\ep f(x) )^2}d\alpha \\[5pt]
              &-2PV\int \frac{\frac{f_x(x)-\frac{f(x)-f(\alpha)}{x-\alpha}}{|x-\alpha|^{2-\ep}}}{(1+(\Delta_{x-\alpha}^\ep f(x) )^2)^2}
              (1+f_x(x)\frac{f(x)-f(\alpha)}{x-\alpha}|x-\alpha|^{2\ep}) d\alpha\\[5pt]
              &+ 3\ep PV \int  \frac{\frac{f_x(x)-\frac{f(x)-f(\alpha)}{x-\alpha}}{|x-\alpha|^{2-\ep}}}{1+(\Delta_{x-\alpha}^\ep f(x) )^2} d\alpha   - 2\ep PV \int\frac{\frac{f_x(x)-\frac{f(x)-f(\alpha)}{x-\alpha}}{|x-\alpha|^{2-\ep}}}{(1+(\Delta_{x-\alpha}^\ep f(x) )^2)^2} d\alpha   \\[5pt]
            \triangleq & I_{11}(x) + I_{12}(x) + I_{13}(x) + I_{14}(x)
\end{split}
\end{equation}
and
\begin{equation}\nonumber
\begin{split}
\pa_x I_2(x)= & \ep \pa_x PV \int \frac{\frac{f(x)-f(x-\alpha)}{|\alpha|^{2-\ep}}}{1+(\Delta_\alpha^\ep f(x))^2} d\alpha \\[5pt]
            = & \ep PV \int \frac{\frac{f_x(x)-f_x(x-\alpha)}{|\alpha|^{2-\ep}}}{1+(\Delta_\alpha^\ep f(x))^2} d\alpha -
            2\ep PV \int \frac{\frac{f_x(x)-f_x(x-\alpha)}{|\alpha|^{2-\ep}}(\Delta_\alpha^\ep f(x))^2}{(1+(\Delta_\alpha^\ep f(x))^2)^2} d\alpha   \\[5pt]
            \triangleq & I_{21}(x) + I_{22}(x).
\end{split}
\end{equation}
Obviously, one has $I_{11}(x_t)=0$. Moreover, one has
\begin{equation}\nonumber
f_x(x_t)-\frac{f(x_t)-f(\alpha)}{x_t-\alpha} \geq 0.
\end{equation}
Without loss of generality, we assume that $\pa_x f(x,t)\leq 0$ holds at least for a short time interval (otherwise one can refer to Remark \ref{SMon}), which is essential in the continuity induction argument.  Then we have
\begin{equation}\nonumber
f_x(x)\frac{f(x)-f(\alpha)}{x-\alpha} |x-\alpha|^{2\ep} \geq 0.\\[10pt]
\end{equation}
Therefore, one has
\begin{equation}\nonumber
I_{12}(x_t) \leq 0, \quad I_{14}(x_t)\leq 0, \quad I_{22}(x_t)\leq 0.
\end{equation}
For $I_{13}(x)$ and $I_{21}(x)$, we use the artificial term $\Lambda^{1-\ep}f_x(x,t)$ to make a balance. This term is given by
\begin{equation}\nonumber
\begin{split}
\Lambda^{1-\ep} f_x(t,x)= & c_1(\ep) \int \frac{f_x(x)-f_x(x-\alpha)}{|\alpha|^{2-\ep}}d\alpha   \\
                        = & c_2(\ep) \int \frac{f_x(x)-\frac{f(x)-f(x-\alpha)}{\alpha}}{|\alpha|^{2-\ep}} d\alpha.
\end{split}
\end{equation}
The two forms (see \cite{Peter}) in the above equalities will be used in different situations. For $I_{13}(x)$, one has
\begin{equation}\nonumber
\begin{split}
   -\frac{C}{2} & \ep \Lambda^{1-\ep}f_x(x,t)+I_{13}(x) \\
= & \ep PV \int \frac{\frac{f_x(x)-\frac{f(x)-f(\alpha)}{x-\alpha}}{|x-\alpha|^{2-\ep}}}{1+(\Delta_{x-\alpha}^\ep f(x))^2}
(3-\frac{C}{2}c_2(\ep))d\alpha \\
 &-\ep\frac{C}{2} c_2(\ep) PV \int \frac{\frac{f_x(x)-\frac{f(x)-f(\alpha)}{x-\alpha}}{|x-\alpha|^{2-\ep}}   (\Delta_{x-\alpha}^\ep f(x))^2 }{1+(\Delta_{x-\alpha}^\ep f(x))^2} d\alpha.
\end{split}
\end{equation}
Choosing $C$ large enough so that $3-\frac{C}{2}c_2(\ep)<0$, one gets
\begin{equation}\nonumber
-\frac{C}{2} \ep \Lambda^{1-\ep}f_x(x,t)+I_{13}(x) \leq 0.
\end{equation}
For $I_{21}(x)$, one has
\begin{equation}\nonumber
\begin{split}
  & -\frac{C}{2} \ep \Lambda^{1-\ep}f_x(x,t)+I_{21}(x) \\
= &  \ep PV \int \frac{\frac{f_x(x)- f_x(x-\alpha) }{|x-\alpha|^{2-\ep}}}{1+(\Delta_{x-\alpha}^\ep f(x))^2}
(1-\frac{C}{2}c_2(\ep))d\alpha  \\
  &-\ep\frac{C}{2} c_2(\ep) PV \int \frac{\frac{f_x(x)- f_x(x-\alpha) }{|x-\alpha|^{2-\ep}}
(\Delta_{x-\alpha}^\ep f(x))^2 }{1+(\Delta_{x-\alpha}^\ep f(x))^2} d\alpha.
\end{split}
\end{equation}
Choosing $C$ large to ensure that $1-\frac{C}{2}c_1(\ep)<0$, one gets
\begin{equation}\nonumber
-\frac{C}{2} \ep \Lambda^{1-\ep}f_x(x,t)+I_{21}(x) \leq 0.
\end{equation}
Therefore, one obtains
\begin{equation}\nonumber
\begin{split}
M\prim (t)= & f_{xt}(x_t,t)  \\
          = & -\ep C \Lambda^{1-\ep}f_x(x_t)+ \ep f_{xxx}(x_t)+ I_x(x_t) \leq 0,
\end{split}
\end{equation}
which means that
\begin{equation}\nonumber
f_x(x,t) \leq f_x(x_t,t)\leq f_x(x_0,0)\leq 0.
\end{equation}
By the standard continuity induction argument, one has
\begin{equation}\nonumber
f_x(x,t) \leq 0.
\end{equation}
By the same method we can deduce that ~$m(t)=\min_{x}f_x(x,t) \geq m(0)$. Thus we have
\begin{equation}\nonumber
\|\pa_x f\|_{L^\infty(\rr)}(t) \leq \|\pa_x f_0\|_{L^\infty(\rr)}.
\end{equation}
Since $f(x,t)$ is monotonically decreasing as the initial data, the $L^\infty$ maximum principle is more or less trivial. One easily finds that ~$b\leq f(x,t)\leq a $. Thus the proof of Lemma \ref{5.2} is completed. 
\end{proof}

\begin{rem}\label{SMon}
If $\pa_x f(x,~t)\leq 0$ does not hold for any short time interval $[0,t_0]$, then we consider $f^\beta(x,~t)=-\beta x+ \widetilde{f}^\beta(x,~t)$ with $0<\beta<1$. By Remark \ref{beta}, it is clear that $\widetilde{f}^\beta(x,t)$ is a solution to system \eqref{beta} with initial data $\widetilde{f}^\beta(x,~0)=f_0(x)$. Since $\pa_x f_0^\beta \leq -\beta$ holds for all $x \in \rr$, and $\widetilde{f}^\beta(x,~t)\in C([0,T],C^{2,\gamma}(\rr)\cap \dot{H}^3(\rr))$ uniformly in $\beta\in(0,1)$, by continuity argument, we have that $\pa_x f^\beta(x,~t)\leq -\frac{\beta}{2}<0$ holds for a time interval $I=[0,t_2]$ for all $\beta \in (0,1)$. On this interval, there are two possibilities:

(i) $|x_t|<\infty$ on a short time interval $[0,t_2]$, where $x_t$ is the argmax of the function $\pa_x f^\beta(x, t)$. In this case, repeating the arguments in the proof of Lemma \ref{5.2}, it is clear that $\pa_x f^\beta(x,t)\leq -\beta$ on $[0,t_2]$; 

(ii) $|x_t|<\infty$ on some time interval $(t_1,t_2]$, $0\leq t_1\leq t_2$, but $|x_t| \rightarrow \infty$ as $t \rightarrow t_1+$. In this case, we know that $\pa_x f^\beta(x,t)$ attains its maximum value at infinity on $[0,t_1]$. Therefore, $\pa_x f^\beta(x,t)\leq -\beta$ holds on $[0,t_1]$. On the time interval $(t_1,t_2]$, Repeating the above argument, we can derive that $M_\beta\prim(t)\leq 0$, which gives\\
\begin{equation}
  \pa_x f^\beta(t_2)\leq \pa_x f^\beta(t), \qquad  \forall ~ t_1< t< t_2.\\[5pt]
\end{equation}
take $t \rightarrow t_1+$, we obtain that $\pa_x f^\beta(x,t) \leq -\beta$ also holds on $[t_1,t_2]$.
Hence, we always have
\begin{equation}
  \pa_x f^\beta(x,t)= -\beta+ \pa_x \widetilde{f}^\beta(x,t) \leq -\beta,  \quad \text{on}~[0,t_2].\\[5pt]
\end{equation}
i.e. $\pa_x \widetilde{f}^\beta(x,t)\leq 0$ holds on a time interval $I$. Taking the limit $\beta\rightarrow 0$, we have $\pa_x f(x,t)\leq 0$ holds on I.
\end{rem}

Taking $\ep\rightarrow 0$, we write system \eqref{OS} in the sense of distribution
\begin{equation}\label{5.4}
\begin{split}
\int_0^T\int_\rr f^\ep(x,t)\phi_t(x,t)&dx dt + \int_\rr f^\ep_0(x) \phi(x,0) dx\\[5pt]
-\ep C \int_0^T \int_\rr  f^\ep &(x,t) \Lambda^{1-\ep}\phi(x,t)dx dt \\[5pt]
&+\ep \int_0^T \int_\rr f^\ep(x,t)\phi_{xx}(x,t)dxdt  \\[5pt]
&-\int_0^T \int_\rr PV \int_\rr \arctan(\Delta_\alpha^\ep f^\ep(x,t))\phi_x(x,t)d\alpha dxdt=0,\\[5pt]
\end{split}
\end{equation}
$\forall \phi(x,t)\in C_c^\infty([0,T)\times \rr).$ Global uniform bounds for $\|f\|_{L^\infty}$ and $\|\pa_x f\|_{L^\infty}$ enable us to take the limit for the linear terms due to the weak-star convergence. The convergence of the nonlinear term
\begin{equation}\nonumber
\begin{split}
&\int_0^T \int_\rr PV \int_\rr \arctan(\Delta_\alpha^\ep f^\ep(x,t))d\alpha \phi_x(x,t)dxdt   \\[5pt]
\longrightarrow & \int_0^T \int_\rr PV \int_\rr
 \arctan(\Delta_\alpha f(x,t))d\alpha \phi_x(x,t)dxdt,   \quad as ~~ \ep \rightarrow 0 
\end{split}
\end{equation}
comes from the \textit{a priori} estimates in Proposition \ref{5.2}, where we can obtain compactness in $L^\infty([0,T]\times B_R)$ and "cut off" the infinity part (details refer to \cite{Peter}).
Note that the initial data converge in the sense of weak topology:
\begin{equation}\nonumber
\int_\rr \JJ \ast f_0(x) \phi(x,0) dx=\int_\rr f_0(x) \JJ \ast \phi(x,0)dx \rightarrow  \int_\rr f_0(x) \phi(x,0)dx,
\end{equation}
$\forall \phi(x,0) \in L^1(\rr).$ Here we have used the approximation of identity in $L^1$, which ends the proof of Theorem \ref{MAINTHM}.

\section*{Acknowledgement}
The first two authors were in part supported by NSFC (grant No. 11421061 and 11222107), National Support Program for Young Top-Notch Talents, Shanghai Shu Guang project, Shanghai Talent Development Fund and SGST 09DZ2272900. The third author was in part supported by NSF grant DMS-1501000.


\frenchspacing
\bibliographystyle{plain}

\end{document}